\documentclass[12pt]{amsart} 
\usepackage{amsmath, amsfonts,amssymb}
\usepackage[usenames]{color} 
\usepackage[all]{xy}     
\usepackage[top=1in, bottom=1in, left=1in, right=1in]{geometry}
\usepackage[colorlinks=true, linkcolor=blue, citecolor=green, urlcolor=cyan, pagebackref, linktocpage=true]{hyperref}

\usepackage{pst-plot}
\usepackage{pst-math}
\usepackage{pst-func}
 
\allowdisplaybreaks
\theoremstyle{plain}

\newtheorem{Teo}{Theorem}[section]
\newtheorem{Lemma}[Teo]{Lemma}
\newtheorem{Prop}[Teo]{Proposition}
\newtheorem{Cor}[Teo]{Corollary}

\theoremstyle{definition}
\newtheorem{Def}[Teo]{Definition}
\newtheorem*{DefTeo}{Theorem/Definition}
\newtheorem*{TeoIntro}{Theorem}
\newtheorem{Rem}[Teo]{Remark}

\newtheorem{Hyp}[Teo]{Hypothesis}
\newtheorem{Ex}[Teo]{Example}

\newtheorem{Notation}[Teo]{Notation}

\numberwithin{equation}{Teo}
\numberwithin{figure}{Teo}

\usepackage{color}   

\newcommand{\cP}{\mathcal{P}}

\newcommand{\cC}{\mathcal{C}}  
\newcommand{\cD}{\mathcal{D}}
\newcommand{\cV}{\mathcal{V}}
\newcommand{\NN}{\mathbb N}  
\newcommand{\ZZ}{\mathbb Z}

\newcommand{\QQ}{\mathbb Q}
 
\newcommand{\inj}{\hookrightarrow}

\newcommand{\Ext}{\operatorname{Ext}}
\newcommand{\Ass}{\operatorname{Ass}}
\newcommand{\Hom}{\operatorname{Hom}}
\newcommand{\Dim}{\operatorname{dim}}

\newcommand{\Ker}{\operatorname{Ker}}
\newcommand{\IM}{\operatorname{Im}}
\newcommand{\Ann}{\operatorname{Ann}}
\newcommand{\InjDim}{\operatorname{inj.dim}}
\newcommand{\Supp}{\operatorname{Supp}}
\newcommand{\Depth}{\operatorname{depth}}
\newcommand{\Rad}{\operatorname{Rad}}

\newcommand{\Spec}{\operatorname{Spec}}

\newcommand{\Length}{\operatorname{length}}
\newcommand{\LengthD}{\operatorname{length}_{D(S,K)}}  
\newcommand{\LengthDshort}{\operatorname{length}_D}

\newcommand{\Char}{\operatorname{char}}
\newcommand{\height}{\operatorname{ht}}
 
\newcommand{\lcm}{\operatorname{lcm}}
\newcommand{\Codim}{\operatorname{codim}} 
\newcommand{\gr}{\operatorname{gr}}
 
\newcommand{\Ord}{\operatorname{ord}}

\newcommand{\itLC}[3]{H^{{#1}_{#2}}_{I_{#2}} \cdots H^{{#1}_2}_{I_2} H^{{#1}_1}_{I_1}\left(#3\right)}
\newcommand{\Mult}{e} 
 
\newcommand{\FDer}[1]{\stackrel{#1}{\longrightarrow}}

\newcommand{\surj}{\twoheadrightarrow}
\newcommand{\deriv}[2]{\frac{\partial^{#2}}{\partial {#1}^{#2}}}
\newcommand{\LC}[3]{H^{#1}_{#2}\left(#3\right)}
\newcommand{\Lyuchar}[1]{\chi_\lambda\left( #1\right)}

\begin{document}  
  
\title{Generalized Lyubeznik numbers}
\author{Luis N\'u\~nez-Betancourt and Emily E.\ Witt}  
\thanks{MSC classes:  13D45, 13N10, 13H99}

\maketitle

\begin{abstract}
Given a local ring containing a field, we define and investigate a family of invariants that includes the Lyubeznik numbers, but that captures finer information. These \emph{generalized Lyubeznik numbers} are defined as lengths of certain iterated local cohomology modules in a category of $D$-modules, and in order to define them,  we develop the theory of a functor Lyubeznik utilized in proving that his original invariants are well defined.   In particular, this functor gives an equivalence of categories with a category of $D$-modules.
These new invariants are indicators of $F$-regularity and $F$-rationality in characteristic $p>0$, and have close connections with characteristic cycle multiplicities in characteristic zero.  We compute the generalized Lyubeznik numbers associated to monomial ideals using interpretations as lengths in a category of straight modules, as well as provide examples of these invariants associated to certain determinantal ideals.


\end{abstract}
\section{Introduction}

The aim of this article is to define and study a family of invariants of a local ring containing a field that includes the Lyubeznik numbers, but that captures finer information.  These invariants are defined in terms of lengths of certain local cohomology modules in a category of $D$-modules.

To prove that these \emph{generalized Lyubeznik numbers} are well defined, we formalize and develop the theory of a functor that Lyubeznik utilized to show that his original invariants are well defined \cite{LyuDmodules}.
In particular, the definition of these new invariants relies heavily on the fact that this functor gives, in fact, a category equivalence with a certain category of $D$-modules.  As a consequence of this new approach, our work also gives a different proof that the original Lyubeznik numbers are well defined. 

Some properties analogous to those of the original invariants hold for the generalized Lyubeznik numbers; however, results on curves and on hypersurfaces show that, unlike the original invariants, the generalized Lyubeznik numbers can differentiate  one-dimensional rings, and complete intersection rings.  

Results of Blickle \cite{Manuel} enable straightforward characterizations of $F$-regularity and $F$-rationality in terms of certain generalized Lyubeznik numbers.
Moreover, recent results of the first author and P\'erez \cite{NuPe} imply that certain generalized Lyubeznik numbers measure how ``far" an $F$-pure hypersurface is from being $F$-regular.
We compute the generalized Lyubeznik numbers associated to monomial ideals
as certain lengths in a category of straight modules, and in characteristic zero, with characteristic cycle multiplicities as well.
The study of the generalized Lyubeznik numbers associated to certain determinantal ideals provides further examples of these new invariants, some striking.  

If $(R,m,K)$ is a local ring admitting a surjection from an $n$-dimensional regular local ring $(S,\eta,K)$ containing a field, and $I$ is the kernel of the surjection, recall that the \emph{Lyubeznik number of $R$ with respect to $i,j \in \NN$}, which depends only on $R$, $i$, and $j$, is defined as $\lambda_{i,j}(R):=\Dim_K \Ext^i_S\left(K,H^{n-j}_I (S)\right)$.  If $(R, m, K)$ is \emph{any} local ring containing a field, we may define $\lambda_{i,j} (R) := \lambda_{i,j} (\widehat{R})$  \cite[Theorem $4.1$]{LyuDmodules}. 
If $d=\dim R$, then $\lambda_{i,j}(R) = 0$ for $j>d$, and $\lambda_{d,d}(R)\neq 0$ \cite[Properties 4.4i, 4.4iii]{LyuDmodules}.  If $R$ is Cohen-Macaulay, $\lambda_{d,d}(R) = 1$ \cite[Theorem 1]{Kawasaki}.  Moreover, the Lyubeznik numbers have extensive geometric and topological interpretations, including connections with \'etale cohomology, and interpretations as the number of connected components of certain punctured spectra.  (See, for example, \cite{B-B, GarciaSabbah, Kawasaki2, Walther2, W}.)

\subsection{Main Results}

A crucial component in proving that the generalized Lyubeznik numbers are well defined is the following equivalence of categories. 

\begin{TeoIntro}[See Theorems \ref{TeoEquiv} and \ref{TeoEquivD}]
Let $R$ be a Noetherian ring, and let $S = R[[x]]$.
Let $\cC$ denote the category of $R$-modules and $\cD$ the category of $D(S,R)$-modules that are supported on $\cV \mathcal(xS)$, the Zariski closed subset of $\Spec (S)$ given by $xS$. Then the functor
\begin{align*}
G: \cC &\to\cD \\
M &\mapsto M\otimes_R S_x/S
\end{align*}
is an equivalence of categories, with inverse functor $\widetilde{G}:\cD\to \cC$ given by $\widetilde{G}(N)=\Ann_N (xS)$. 

Moreover, if $R=K[[y_1,\ldots,y_n]]$, $K$ a field, then $S=K[[y_1,\ldots,y_n,x]],$ and $G$ is an equivalence of categories between the category of
$D(R,K)$-modules and the category of $D(S,K)$-modules supported on $\cV \mathcal(xS).$
\end{TeoIntro}

To define the generalized Lyubeznik numbers, we depend on the fact that certain local cohomology modules have finite length as 
$D(S,K)$-modules
(cf.\ Section \ref{D-modules}), 
where $K$ is a field and $S=K[[x_1,\ldots,x_n]]$ for some $n$ \cite[Corollary 6]{Lyu2}.
These new invariants depend on the local ring $R$ containing a field, a collection of ideals $I_1,\ldots, I_s$ of $R$, as well as $j_1,\ldots,j_s\in \NN$.   The definition is as follows. 

\begin{DefTeo}[See Theorem \ref{WellDef}, Definition \ref{NewLyuNums}]
Let $(R,m,K)$ be a local ring containing a field, so that the completion $\widehat{R}$ of $R$ at $m$ admits a surjective ring 
map $\pi : S\surj \widehat{R}$, where $S=K[[x_1,\ldots, x_n]]$ for some $n$.
Fix $I_1,\ldots, I_s$ ideals of $R$ and $i_1,\ldots, i_s \in \NN$.
If $J_1,\ldots, J_s$ denote the corresponding preimages of $I_1 \widehat{R},\ldots, I_s \widehat{R}$ 
in $S$, then the \emph{generalized Lyubeznik number of $R$ with respect to $I_1,\ldots, I_s$ and $i_1,\ldots, i_s$} is defined as $$\lambda^{i_s,\ldots, i_1}_{I_s,\ldots, I_1}(R) := \Length_{D(S,K)} H^{i_s}_{J_s}   \cdots H^{i_2}_{J_2}H^{n-i_1}_{J_1}(S).$$  
Moreover, $\lambda^{i_s,\ldots, i_1}_{I_s,\ldots, I_1}(R)$ is finite and depends only on $R$, $I_1,\ldots, I_s$, and $i_1,\ldots, i_s$, but neither on $S$ nor on $\pi$.
\end{DefTeo}

We show that the family of generalized Lyubeznik numbers does, in fact, contain the original Lyubeznik numbers (see Proposition \ref{PropRmk}).  As a consequence of this proof, we give a new proof of the fact that the Lyubeznik numbers are well defined.

We prove generalizations of some vanishing results for the original invariants (see Proposition \ref{FirstProp}).  We also investigate the behavior of the generalized Lyubeznik numbers under finite field extensions, as well as derive an inequality of the generalized Lyubeznik numbers with characteristic cycle multiplicity in characteristic zero (see Propositions \ref{FieldExtension} and \ref{CCineq}).

Unlike the original Lyubeznik numbers, results on curves and on hypersurfaces show that the new invariants can differ for one-dimensional and complete intersection rings (see Propositions \ref{PropCurves} and \ref{PropHyps}), confirming that the generalized Lyubeznik numbers capture finer information than do the original ones.  We also define a new invariant, the \emph{Lyubeznik characteristic}, in terms of certain generalized Lyubeznik numbers (see Definition \ref{LyuCharc}).

We investigate further properties of the generalized Lyubeznik numbers in several of cases.  We point out characterizations of $F$-regularity and $F$-rationality in terms of these invariants that follow from work of Blickle (see Proposition \ref{F-regularity} and Corollary \ref{F-regularityGor}) \cite{Manuel}.  We use this, and recent results of the first author and P\'erez to point out that certain generalized Lyubeznik numbers measure how ``far" an $F$-pure hypersurface is from being $F$-regular (see Remark \ref{chains}).

We also compute certain generalized Lyubeznik numbers corresponding to ideals of maximal minors (see Section \ref{MaxMinors}).  In particular, these give a striking illustration of the generalized Lyubeznik numbers' strong characteristic dependence (see Remark \ref{CharDep}).   

Finally, we study the generalized Lyubeznik numbers corresponding to monomial ideals.  In particular, it is possible to compute these invariants as certain lengths in the category of straight modules, and in terms of characteristic cycle multiplicities in characteristic zero (see Theorem \ref{sqfreeExt}).  Using work of \`Alvarez-Montaner, we bound certain generalized Lyubeznik numbers in terms of the minimal primes of the corresponding monomial ideal \cite{Montaner}.  We also compute the Lyubeznik characteristic of Stanley-Reisner rings in terms its faces of
(see Theorem \ref{LyuCharSR});
Interestingly, this invariant is characteristic independent in this case, even though the original Lyubeznik numbers are not (see Remark \ref{LyuCharChar}).

\subsection{Outline}  
Section \ref{Prelim} gives relevant background on $D$-modules (\ref{D-modules}) and on positive characteristic methods (\ref{PosChar}).  In Section \ref{KeyFunctor}, we develop the theory of a functor that Lyubzenik used to show that the original Lyubeznik numbers are well defined \cite{LyuDmodules}.  Theorem \ref{TeoEquiv}, and Theorem \ref{TeoEquivD} in Section \ref{DefFirstProp}, show that this functor gives an equivalence of categories with a category of $D$-modules.  In Section \ref{DefFirstProp}, the results on this functor are critically used to define the generalized Lyubeznik numbers (see Theorem \ref{WellDef} and Definition \ref{NewLyuNums}); Proposition \ref{PropRmk}  shows that these invariants include the original Lyubeznik numbers.  
Here, we also give some properties of  the generalized Lyubeznik numbers and define the \emph{Lyubeznik characteristic}.
Section \ref{Frelations} states interpretations of $F$-rationality and $F$-regularity in terms of certain generalized Lyubeznik numbers through results of Blickle \cite{Manuel}.  In Section \ref{MaxMinors}, we give examples of generalized Lyubeznik numbers corresponding to the maximal minors of a generic matrix.  Finally, in Section \ref{Monomial}, we compute these invariants associated to monomial ideals using the theory of straight modules, and in terms of characteristic cycle multiplicities in characteristic zero.  We also compute the Lyubeznik characteristic of  Stanley-Reisner rings.

\section{Preliminaries} \label{Prelim}

\subsection{$D$-modules} \label{D-modules}

Given rings $A\subseteq S$,  we define the ring of $A$-linear differential operators of $S$, $D(S,A)$, as 
the subring of $\Hom_A(S,S)$ defined inductively as follows:
the differential operators of order zero are induced by multiplication by elements in $S$.  An element $\theta \in \Hom_A(S,S)$ is a differential operator of order less than or equal to $k+1$ if, for every $r \in S$, $[\theta,r]:=\theta\cdot r -r \cdot\theta$
is a differential operator of order less than or equal to $k$. 
From the definition, we have that if $B$ is a subring $A,$ we have that $D(S,A)\subseteq D(S,B).$

If $M$ is a $D(S,A)$-module, then $M_f$ has the structure of a $D(S,A)$-module such that, for every $f \in S$, the natural
morphism $M\to M_f$ is a morphism of $D(S,A)$-modules. 
As a result, since $S$ is a $D(S,A)$-module, for 
all ideals $I_1,\ldots,I_s\subseteq S$,
and all $i_1,\ldots i_s \in \NN$,
$\itLC{i}{\ell}{S}$ is also a $D(S,A)$-module \cite[Example 2.1(iv)]{LyuDmodules}.

By \cite[Theorem $16.12.1$]{EGA},
if $S=A[[x_1,\ldots,x_n]]$, 
then $$D(S,A)=S\left\langle\frac{1}{t!} \deriv{x_i}{t} \ | \ t \in \NN, 1 \leq i \leq n \right\rangle\subseteq \Hom_A(S,S).$$ 
Moreover, if $A=K$ is a field, then $S_f$ has finite length in the category of
$D(S,K)$-modules for every $f \in S$. Consequently, every module of the form $\itLC{i}{s}{S}$ also has finite
length in this category \cite[Corollary 6]{Lyu2}.

\begin{Hyp} \label{hyp}
Throughout the rest of Section \ref{D-modules}, we will assume that $S$ is either or  $K[x_1,\ldots,x_n]$ or $K[[x_1,\ldots,x_n]]$, where $K$ is a field of characteristic $0$. Let $D = D(S,K)$.
\end{Hyp}

We recall some relevant definitions and properties of $D$-modules, and refer the reader to \cite{Bj1,Bj2,Coutinho,MeNa} for details.
Under Hypothesis \ref{hyp}, we know that $D=S\left\langle\frac{\partial}{\partial x_1},\ldots,\frac{\partial}{\partial x_n}\right\rangle\subseteq \Hom_K(S,S)$, and there is an ascending filtration 
$$
\Gamma_i:=\{\delta\in D \mid  \Ord(\delta)\geq i\}
=\bigoplus_{\alpha_1+\ldots +\alpha_n\leq i} R\cdot\frac{\partial^\alpha}{\partial x_i^\alpha}.
$$
Moreover,  $\gr^\Gamma(D)\cong S[y_1,\ldots,y_n]$, a polynomial ring over $S$.
A filtration $\Omega=\{\Omega_j\}$ of $S$-modules on a $D$-module $M$ is a \emph{good filtration}
if  $\Omega_j\subseteq \Omega_{j+1}$, $\bigcup \limits_{j\in\NN} \Omega_{j}=M$, $\Gamma_i \Omega_j\subseteq \Omega_{i+j}$, and $\gr^{\Omega}(M)=\bigoplus \limits_{j\in \NN} \Omega_{j+1} / \Omega_i$ is a finitely generated $\gr^{\Gamma}\left(D(S,K)\right)$-module.


If $\Gamma$ is a good filtration,
neither $\dim_{\gr^\Gamma(D)} \gr^{\Omega}(M) $
nor $\Rad(\Ann_{\gr^{\Gamma}(D)}\gr^{\Omega}(M))$ depend on the choice of good filtration. For the sake of clarity, 
we will omit the filtration when referring to the associated graded ring or module. 

A finitely generated $D$-module $M$ is \emph{holonomic} if either
$M=0$ or
$\dim_{\gr(D)} \gr(M)=n$. The holonomic $D$-modules form a full abelian subcategory of the category
of $D$-modules, and every holonomic $D$-module has finite length as a $D$-module. Moreover, if $M$ is holonomic, then $M_f$ is also holonomic for every $f\in S$.
As a consequence, since $S$ is holonomic, every module of the form $\itLC{i}{\ell}{S}$ is also.

\begin{Def}[Characteristic variety, characteristic cycle, characteristic cycle multiplicity] \label{CCdef}
Given a holonomic $D$-module, the \emph{characteristic variety of $M$} is
$$
C(M)=\cV\left(\Rad\left(\Ann_{\gr \left(D(S,K)\right)}\gr(M)\right)\right)\subseteq \Spec\gr(D),
$$
and its \emph{characteristic cycle} is
$
CC(M)=\sum m_i V_i,
$
where the sum is taken over all the irreducible components $V_i$ of $C(M)$, and $m_i$
is the corresponding multiplicity. We define the \emph{\textup{(}characteristic cycle\textup{)} multiplicity of $M$} by
$
e(M)= \sum m_i.
$ 
\end{Def}
\begin{Rem}\label{PropMultiplicity}
 If $0\to M'\to M\to M''\to 0$ is an exact sequence of holonomic $D$-modules, then
$CC(M)=CC(M') + CC(M'')$; as a consequence,
$e(M)=e(M')+e(M'')$. In addition, $CC(M)=0$ if and only if $M =0$, so that $e(M)=0$ if and only if $M=0$ as well. 
\end{Rem}

Now let $S=K[x_1,\ldots,x_n]$, and take $f\in S$.  Let $N[s]$ be the  free $S_f[s]$-module generated by a symbol ${\bf f^s}$. We give $N[s]$ a
left $D_f[s]$-module structure as follows:
$$
\frac{\partial}{\partial x_i} \cdot \frac{g}{f^\ell}{\bf f^s}=
\left(\frac{1}{f^\ell}\frac{\partial g}{\partial x_i} -s \frac{g}{f}\frac{\partial f}{\partial x_i}  \right){\bf f^{-s}}.
$$
There exist a polynomial $0 \neq b(s)\in \QQ[s]$ and an operator $\delta(s)\in D[s]$ that satisfy
\begin{equation}\label{BS}
\delta(s)f\cdot(1\otimes {\bf f^s})=b(s)(1\otimes {\bf f^s})
\end{equation} 
in $N[s]$ \cite[Chapter $10$]{Coutinho}.

Given $\ell \in \ZZ$, we define the specialization map $\phi_\ell : N[s]\to R_f$ by 
$\phi_\ell (v s^i\otimes {\bf f^s})=\ell^i  v {\bf f^{\ell}}.$ Thus, $\phi_\ell (\delta(s)v)= \delta(\ell)\phi_\ell (v)$.
Then, by applying this morphism to the result, we have
$$
\delta(\ell)  f^{\ell+1} =b(\ell) f^\ell. 
$$

The set of all polynomials $h(s)\in\QQ[s]$ that satisfy Equation \ref{BS} forms an ideal of $\QQ[s].$ 
We call the minimal monic polynomial satisfying it the
\emph{Bernstein-Sato polynomial of $f$}, and denote it $b_f(s).$ 

\subsection{Methods in positive characteristic} \label{PosChar}

We briefly recall several methods used in the study of rings of positive characteristic.  Our summary is based on 
\cite{Fedder} for $F$-purity and $F$-injectivity, 
\cite{P-S} for the Frobenius functor, 
\cite{Manuel-DRF-Modules} for $R[F^e]$-modules, 
\cite{Ye} for $D$-modules in positive characteristic, and  \cite{HoHu1, HoHu2,Smith-PI} for tight closure. We refer the reader to these articles for details.

Throughout this section, $R$ is a ring of characteristic $p>0$ and 
$F:R\to R$ denotes the Frobenius morphism, $r\mapsto  r^p.$ 
If $R$ is reduced, we define $R^{1/q}$ as the ring of formal $q^\text{th}$-roots of $S$.
A ring $R$ is \emph{$F$-finite} if $R^{1/p}$ is a finitely generated $R$-module.

We say that $R$ is \emph{$F$-pure} if for every $R$-module $M$, 
the morphism induced by the inclusion of $R \inj R^{1/p}$, $M\otimes_R R\to M\otimes_R R^{1/p}$,
is injective. 
If $M$ is an $R$-module, then $F$ acts naturally on it. If $(R,m,K)$ is local, we say that a ring is  $F$-injective if the induced Frobenius map
$F: H^i_m (R) \to H^i_m (R)$ is injective for every $i\in\NN$. $F$-purity implies $F$-injectivity, and  in a Gorenstein ring,
these properties are equivalent \cite[Lemma 3.3]{Fedder}.

An \emph{$R[F^e]$-module} is an $R$-module $M$ with an $R$-linear map $\nu^e : F^{e*} M\to M.$  If $\nu^e$ is an isomorphism, then $(M, \nu^e)$ is called a \emph{unit $R[F^e]$–-module.} By adjointness there is a one-to-one correspondence between maps $\nu^e_M \in Hom(F^{e*}M,M)$ and   maps $F^e_M\in  \Hom(M, F^{e*}M),$ where $F^e_M (u) = \nu^e_M (1 \otimes u)$. An element $u \in M$ of an $R[F^e]$-module $(M, \nu^e)$ is called \emph{$F$-nilpotent} if $F^{e\ell }(u) = 0$ for some $\ell\in\NN;$  $M$ is called \emph{$F$–-nilpotent} if $F^{e\ell} (M) = 0.$

If $R$ is a reduced $F$-finite ring, then
$D(R,\ZZ)=\bigcup_{e\in\NN} \Hom_{R^{p^e}}(R,R).$
Moreover, if $K$ is a perfect field and $R=K[[x_1,\ldots,x_n]]$, then $D(R,\ZZ)=D(R,K)$.

If $I$ is an ideal of $R$, the \emph{tight closure} $I^*$ of $I$ is the ideal of $R$
consisting of all those elements $z\in R$ for which there exists some $c \in R$, $c$
not in any minimal prime of $R$, such that
$cz^q \in I^{[q]} \hbox{ for all }q = p^e \gg 0$, where $I^{[q]}$ denotes the ideal of $R$ generated by $q^\text{th}$ powers of elements in $I$.

We say that $R$ is \emph{weakly $F$-regular} if $I=I^*$ for every ideal $I$ of $R$.
If every localization of $R$ is weakly $F$-regular, then $R$ is \emph{$F$-regular}.
In general, tight closure does not commute with localization, and it is unknown whether the localization of 
a weakly $F$-regular ring must again be weakly $F$-regular; this explains the use of the adjective ``weakly."
If $R$ is a local ring, we say that the ring is $F$-rational if for every parameter ideal $I$, $I=I^*.$

A ring $R$ is \emph{strongly $F$-regular} if for all $c\in R$ not in any minimal prime,
there exists some $q=p^e$ such that the $R$-module map $R \to R^{1/q}$ sending $1\mapsto c^{1/q}$ splits. Strong $F$-regularity is preserved under localization.
In a Gorenstein ring, $F$-rationality, strong $F$-regularity, and
weak $F$-regularity are equivalent. 



Given a Noetherian ring $R$ of prime characteristic $p>0$, if $N \subseteq M$ is an inclusion of $R$-modules, then the \emph{tight
closure  $ N^*_M$ of $N$ in $M$} consists of all elements $u \in M$, such that for some $c$ not in any minimal prime of $R,$
$ cu^q \in N^{[p^e]}_M := \IM\left( F^e(N) \to F^e(M)\right) \subseteq F^{e}(M) \textrm{ for all  } p^e\gg 0.$
\section{A Key Functor} \label{KeyFunctor}
In this section, we study a functor utilized by Lyubeznik to prove that his original invariants are well defined (cf.\ \cite[Lemma $4.3$]{LyuDmodules}). 
In order to prove that the generalized Lyubeznik numbers are well defined, significant development of the theory of this functor is necessary.  The fact that this functor gives, in fact, an equivalence with a certain category of $D$-modules is essential to the results here, as we will see in Theorem \ref{TeoEquiv}.


\begin{Def}[Key functor $G$] \label{Gfunctor}
Let $R$ be a Noetherian ring, and let $S = R[[x]]$.  
Let $G:R\operatorname{-mod}\to S\operatorname{-mod}$ be the functor given by $G(-)= ( - ) \otimes_R S_x/S$. 
\end{Def}

\noindent We note that the functor $G$ is reminiscent of the ``direct image" functor utilized by \`Alvarez Montaner, by differs due to the base ring in the tensor product \cite{AM-Proc}.

\begin{Rem}
For every element in $u\in G(M)$ there exist  $\ell, \alpha_1,\ldots, \alpha_\ell\in \NN$, $m_1,\ldots,m_\ell\in M$, uniquely determined, such that $u=m_\ell \otimes x^{-\alpha_\ell}+\ldots + m_1 \otimes x^{-\alpha_1}$ and $m_\ell\neq 0$ because \begin{equation} \label{directsumdecomp} G(M)=M \otimes_R S_x/S = M \otimes_R \left(\bigoplus \limits_{\alpha \in \NN} R x^{-\alpha}\right) = \bigoplus \limits_{\alpha \in \NN} \left(M \otimes R x^{-\alpha}\right).  \end{equation} Moreover, $G$ is an exact functor and commutes with local cohomology.

\end{Rem}

\begin{Rem}
In fact, $G$ is a functor from the category $R$-modules to the 
category of $D(S,R)$-modules:
Let $M$ be a $D(S,R)$-module. 
Since  
$D(S,R)=S \langle \frac{1}{t!} \deriv{x}{t} \ | \ t \in \NN \rangle \subseteq \Hom_K(S,S)$, it is enough to give an action of each
$\frac{1}{t!} \deriv{x}{t}$ on $G(M)$.
If $m\otimes x^{-\alpha} \in G(M)$, we define
\begin{equation*} \label{daction} \left(\frac{1}{t!}  \deriv{x}{t} \right)\cdot (m\otimes x^{-\alpha})=  \binom{\alpha+t-1}{t} \cdot \left((-1)^{t} m\otimes
x^{-\alpha-t}\right). \end{equation*} 
In particular, taking $\alpha = $1 and $t = \beta$, we see that,  for every $\beta \in \NN$, \begin{equation}  \label{daction2}   m\otimes x^{-\beta}
=\frac{(-1)^{\beta-1}}{(\beta-1)!}
 \deriv{x}{\beta-1}
(m \otimes x^{-1}). \end{equation}

Similarly, for every morphism of $R$-modules $\varphi$,
$G(\varphi)=\varphi \otimes_R  S_x/S$ is a morphism of $D(S,R)$-modules.
\end{Rem}

Moreover, $G$ is an equivalence of certain categories:

\begin{Teo}\label{TeoEquiv}
Let $R$ be a Noetherian ring, and let $S = R[[x]]$.
Let $\cC$ denote the category of $R$-modules and $\cD$ denote the category of $D(S,R)$-modules that are supported on $\cV \mathcal(xS)$, the Zariski closed subset of $\Spec (S)$ given by $xS$. Then
$G:\cC\to\cD$ as in Definition \ref{Gfunctor} is an equivalence of categories with inverse functor $\widetilde{G}:\cD\to \cC$ given by $\widetilde{G}(M)=\Ann_M (xS)$.
\end{Teo}

\begin{proof}
It is clear that for every $R$-module $M$, $\widetilde{G}(G(M))$ is naturally isomorphic to $M$.
It suffices to prove that for every $D(S,R)$-module $N$ with support on $\cV (xS)$, $G(\widetilde{G}(N))$ is naturally isomorphic to $N$.  
Let $M=\widetilde{G}(N)=\Ann_N (xS)$, and
let $\phi:G(M)\to N$ be the morphism of $R$-modules defined on simple tensors by $m\otimes x^{-\alpha}\mapsto \frac{(-1)^{\alpha-1}}{(\alpha-1) !}  \deriv{x}{\alpha-1} m$. 
We will prove, in steps, that $\phi$ is an isomorphism of $D(S,R)$-modules.

First, we will show that $\phi$ is a morphism of $D(S,R)$-modules. 
Since $D(S,R)=S\langle\frac{1}{t!}  \deriv{x}{t} \ | \ t \in \NN \rangle$, it is enough to show that $\phi$ commutes with multiplication by $x$ and by any 
$\frac{1}{t!} \deriv{x}{t}$.

We first prove commutativity with $\frac{1}{t!} \deriv{x}{t}$.  For any $t \in \NN$,
\begin{align*} 
\phi\left(\frac{1}{t!}\deriv{x}{t}  (m\otimes x^{-\alpha})\right)&=
\phi\left(\binom{\alpha+t-1}{t}  \left( (-1)^t m\otimes x^{-\alpha-t} \right) \right)\\
&=
\binom{\alpha+t-1}{t}  \frac{(-1)^{\alpha - 1}}{(\alpha+t-1)!}\deriv{x}{\alpha+t-1} m\\
&=
\frac{1}{t!} 
\frac{(-1)^{\alpha-1}}{(\alpha-1)!}
\deriv{x}{\alpha+t-1}
 m\\
&= \frac{1}{t!}
\deriv{x}{t} 
\left( \frac{(-1)^{\alpha-1}}{(\alpha-1)!}
\deriv{x}{\alpha-1}
m \right) \\
&=\frac{1}{t!} \deriv{x}{t} \phi( m\otimes x^{-\alpha}),
\end{align*} 
which is sufficient.

We now prove that the morphism commutes with $x$. 
Note that
$$
x\frac{1}{t !}\deriv{x}{t} - \frac{1}{t !}\deriv{x}{t}x
=-\frac{1}{(t-1) !}\deriv{x}{t-1}
$$
as differential operators for every $t\in\NN$. We conclude that
\begin{align}
\phi(x (m\otimes x^{-\alpha})) =\phi(m\otimes x^{-\alpha+1}) &= \phi\left(m\otimes  \frac{(-1)^{\alpha-2}}{(\alpha-2) !}\deriv{x}{\alpha-2} x^{-1}\right) \notag \\
&=\frac{(-1)^{\alpha-2}}{(\alpha-2) !}\deriv{x}{\alpha-2} \phi(m\otimes x^{-1})  \label{three} \\
&=x\frac{(-1)^{\alpha-1}}{(\alpha-1) !}\deriv{x}{\alpha-1}\phi(m\otimes x^{-1}) - 
\frac{(-1)^{\alpha-1}}{(\alpha -1) !}\deriv{x}{\alpha-1} x \phi(m\otimes x^{-1})  \notag \\
&=x\frac{(-1)^{\alpha-1}}{(\alpha-1) !}\deriv{x}{\alpha-1}\phi (m\otimes x^{-1})  \notag \\
&=x\phi (m\otimes  \frac{(-1)^{\alpha-1}}{(\alpha-1) !}\deriv{x}{\alpha-1} x^{-1})  \label{six} \\
&=x\phi( m\otimes x^{-\alpha}),  \notag
\end{align}
where \eqref{three} and \eqref{six} are due to the commutativity of $\frac{1}{t!}\deriv{x}{t}$.  

It remains to prove that $\phi$ is bijective; we proceed by contradiction. Suppose that there exists $u=m_\ell \otimes 
x^{-\alpha_\ell}+\ldots + m_1 \otimes x^{-\alpha_1}\in \Ker (\phi)$  such that $m_\ell\neq 0$. 
Then $\phi(m_\ell\otimes x^{-1})=\phi(x^{\ell-1}u)=x^{\ell-1} \phi(u)=0$. Thus, $m_\ell=0$ because
$\phi|_{M\otimes R x^{-1}}$ is bijective, and we get a contradiction.

We now see that $\phi(\Ann_{G(M)} (x^jS))=\Ann_{N} (x^jS)$ for every $j\geq 1$ by induction, which will imply that $\phi$ is 
surjective (since $N$ is supported on $\cV(xS)$).  Since $\phi(\Ann_{G(M)} (x^jS)) \subseteq \Ann_{N} (x^jS)$ for all $j$, we seekthe opposite inclusion.  
For $j=1$, take $n \in M=\Ann_N (xS)$; then $n \otimes x^{-1} \in G(M)$, so $\phi(n \otimes x^{-1}) = n$.
Now take any $j \geq 1$ and assume the statement holds for $j-1$. For any $u\in \Ann_{N} (x^jS)$,
 $xu\in \Ann_{N} (x^{j-1}S)$, so $xu=\phi(v)$ for some $v=m_{j-1}\otimes x^{-j+1}+\ldots + m_1 \otimes x^{-1}\in G(M)$
 by the inductive hypothesis. Let $w=m_{j-1}x^{-j}+\ldots + m_1 \otimes x^{-2}$.
Thus, $x\phi(w)=\phi(x w)=\phi(v)=xu$. This means that $x (\phi(w)-u)$ = 0, and so $\phi(w)-u\in \Ann_N (xS)=\phi(\Ann_{G(M)} (xS))$ and 
$\phi( m'\otimes x^{-1})=\phi(w)-u$ for some $m'\in M$ by the base case. 
Therefore, $u=\phi(w-m\otimes x^{-1})\in \phi(\Ann_{G(M)}(x^{j}S))$. 
\end{proof}
\begin{Prop}\label{PropFG}
Let $R$ be a Noetherian ring, and let $S = R[[x]]$. 
Then $M$ is a finitely generated $R$-module if and only if $G(M)$ is a finitely generated $D(S,R)$-module.

\end{Prop}
\begin{proof}
Given $m_1,\ldots,m_s\in M$, generators for $M$ as $R$-module, 
 $m_1 \otimes x^{-1} ,\ldots,m_s\otimes x^{-1}$
generate $G(M)$ as a $D(S,R)$-module: by \eqref{daction}, for $\beta \in \NN$,
$m_i\otimes x^{-\beta}
=\frac{(-1)^{\beta-1}}{(\beta-1)!}
\deriv{x}{\beta-1}
(m_i \otimes x^{-1})$, and the set
$\{ m_i\otimes x^{-\beta} \ | \ 1 \leq i \leq s, \beta \in \NN \}$ generates $G(M)$ as an $R$-module.

If $u_1, \ldots, u_s\in G( M)$ is a set generators for $G(M)$ as a $D(S,R)$-module, then each $u_i$ can be written as 
$u_i = m_{i,1}\otimes x^{-1} + m_{i,2} \otimes x^{-2}  + \ldots + m_{i,\ell_i}\otimes x^{-\ell_i}$ for some $\ell_i \in \NN$ and $m_{i,j} \in M$. 
Then $\{ m_{i,j}\otimes x^{-j} \ | \ 1 \leq i \leq s, 1 \leq j \leq \ell_i\}$ is also a set of generators for $G(M)$ as a $D(S,R)$-module.  Since $m_{i,j}\otimes x^{-j}=\frac{(-1)^{j-1}}{(j-1)!}\deriv{x}{j-1}
(m_{i,j}\otimes x^{-1})$, the decomposition in \eqref{directsumdecomp} implies that the $m_{i,j}$ must generate $M$.
\end{proof}

\begin{Cor}\label{CorLen}
Let $R$ be a Noetherian ring, $M$ an $R$-module, and $S = R[[x]]$. 
Then $\Length_R (M)=\Length_{D(S,R)} G(M)$.
\end{Cor}
\begin{proof}
If $M$ is a simple nonzero $R$-module, then $G(M)$ is a simple $D(S,R)$-module since
the $D(S,R)$-submodules of $G(M)$ correspond precisely to $R$-submodules of $M$ by Theorem \ref{TeoEquiv}.
Now say that $\Length_{R}(M)=h< \infty$, so that we have a filtration of $R$-modules $0 = M_0\subsetneq M_1 \subsetneq \ldots \subsetneq M_h=M$
such that each $M_{j+1}/M_j$ is a
simple $R$-module. Then $0 = G(M_0)\subseteq G(M_1) \subseteq \ldots \subseteq G(M_h)=G(M)$ is a filtration of $D(S,R)$-modules such that
$G(M_{j+1})/G(M_j)\cong G(M_{j+1}/M_j)$ is a simple $D(S,R)$-module for every $j$ by our initial argument. Therefore, $\Length_{D(S,R)}(G(M)) = h$. 
Similarly, if $\Length_{R}(M)$ $=\infty$, then $\Length_{D(S,R)}(G(M))=\infty$.  
\end{proof} 
\begin{Rem}
In the following work, we often make use of the following observation:
for $R$ a ring and $S=R[[x]]$,
if $P$ is a prime ideal of $R$, then $(P,x)S$ is a prime ideal of $S$ since $S/(P,x)S=R/P$ is a domain.
\end{Rem}

\begin{Prop}
Let $R$ be a Noetherian ring, $M$ an $R$-module, and $S=R[[x]]$.  Then
$
\Ass_S G(M)=\{(P,x)S\mid P\in \Ass_R M\}.
$
\end{Prop}
\begin{proof}
Let $Q\in\Ass_S G(M)$, so that $Q=\Ann_S u$ for some $u\in G(M)$. As $H^0_{xS}\left(G(M)\right)=G(M)$, $x\in Q$. Thus, 
$u\in \Ann_{G(M)} xS\cong M$ (the isomorphism is due to Theorem \ref{TeoEquiv}). Moreover, we have the natural epimorphism $R\surj S/Q$ with kernel $P=\Ann_R u \in \Ass_R M$.  Thus, $Q=(P,x)S$. 
 
Take $Q = (P,x)S$, where $P = \Ann_R u \in \Ass_R M$, $u \in M$.  
Then $Q=\Ann_S (u\otimes x^{-1})$. Hence, $Q\in \Ass_S G(M)$. 
\end{proof}

\begin{Lemma}\label{LemmaIdeal} 
Let $R$ be a Noetherian ring, $M$ an $R$-module, and $S= R[[x]]$. 
Then for every ideal $I\subseteq R$ and all $j\in \NN$, $G\left(H^j_I(M)\right)=H^{j+1}_{(I,x)S}(M\otimes_R S)$.
\end{Lemma}
\begin{proof}
Since $S$ and $S_x$ are flat $R$-algebras and $S_x/S$ is a free $R$-module, we know that
$H^j_I(M)\otimes_R S=H^j_{IS}(M\otimes_R S)$, $H^j_I(M)\otimes_R S_x=H^j_{IS}(M\otimes_R S_x)$ and 
$H^j_I(M)\otimes_R S_x/S=H^j_{IS}(M\otimes_R S_x/S)$. Moreover,  the sequence
\begin{equation} \label{exseq1}
0\to H^j_{IS}(M\otimes_R S)\to H^j_{IS}(M\otimes_R S_x)\to H^j_{IS}(M\otimes_R S_x/S) \to 0
\end{equation}
is exact, so $G(H^j_{I}(M)) = H^j_{IS}(M\otimes_R S_x) / H^j_{IS}(M\otimes_R S).$
 
On the other hand, we have a long exact sequence
\begin{equation*}
\cdots\to H^j_{(I,x)S}(M\otimes_R S)\to H^j_{IS}(M\otimes_R S)\to H^j_{IS}(M\otimes_R S_x) \to \cdots.
\end{equation*}
Since $H^j_{IS}(M\otimes_R S)\to H^j_{IS}(M\otimes_R S_x)$ is injective by \eqref{exseq1}, the long sequence splits into
short exact sequences 
\begin{equation*}
0 \to H^j_{IS}(M\otimes_R S)\to H^j_{IS}(M\otimes_R S_x) \to H^{j+1}_{(I,x)S}(M\otimes_R S)\to 0.
\end{equation*}
Hence, $G\left(H^j_{I}(M)\right)=H^{j+1}_{(I,x)S}(M\otimes_R S )$.
\end{proof}

\begin{Prop}\label{PropIdeals} 
Let $(R,m,K)$ be a Noetherian local ring, $M$ an $R$-module, and $S=R[[x]]$. 
Fix $I_1,\ldots,I_s$ ideals of $R$ and $j_1,\ldots j_s \in \NN$. 
Then  \begin{small} $$G\left(\itLC{j}{s}{M} )\right) \cong
H^{j_s}_{(I_s,x)S} \cdots H^{j_2}_{(I_2,x)S}H^{j_1+1}_{(I_1,x)S}(M\otimes_R S).$$ \end{small}
\end{Prop}
\begin{proof}
We proceed by induction on $s$.
If $s=1$, the statement follows from Lemma \ref{LemmaIdeal}. Suppose it holds for some $s\geq1$.
Let $N_\ell =H^{j_\ell }_{I_\ell }    \ldots  H^{j_2}_{I_2}H^{j_1}_{I_1}(M)$
for $1\leq \ell\leq s+1$, so we need to prove that 
$G(N_{s+1}) \cong H^{j_{s+1}}_{(I_{s+1},x )S} (G(N_s))$.
Now, $$G(N_{s+1})=H^{j_{s+1}}_{I_{s+1}} (N_s)\otimes_R S_x/ S
\cong H^{j_{s+1}}_{I_{s+1}S} (N_s\otimes_R S_x/ S)
=
H^{j_{s+1}}_{I_{s+1} S} (G(N_s)).
$$
Consider the long exact sequence of functors
\begin{equation}
\ldots\to H^{j_{s+1}}_{I_{s+1}S} (-)\to H^{j_{s+1}}_{(I_{s+1},x)S}(-)\to H^{j_{s+1}}_{I_{s+1}S}(-\otimes_S S_x) \to \ldots.
\end{equation}
Since $G(N_s)$ is supported on $\cV(xS)$,  $H^{i}_{I_{s+1} S} (G(N_s)\otimes_S S_x)=0$ for all $i\in \NN$, and $G(N_s)\otimes_S S_x = 0$. 
Moreover,
$H^{j_{s+1}}_{I_{s+1} S} (G(N_s))\cong H^{j_{s+1}}_{(I_{s+1},x) S} (G(N_s)).$
Hence, $G(N_{s+1})\cong H^{j_{s+1}}_{(I_{s+1},x )S} (G(N_s))$.
\end{proof}

As $G$ is an equivalence of categories, $G(\Hom_R(M,N))=\Hom_{D(S,R)}(G(N),G(M))$. Thus, $M$ is an injective $R$-module if and only if $G(M)$ is an injective  object in $\cD$, the category of $D(S,R)$-modules supported at $\cV(xS)$. We now characterize precisely when $G(M)$ is injective as an $S$-module:
 
\begin{Prop}\label{PropInj}
Let $S=R[[x]]$, where $R$ is a Gorenstein ring.
Given a prime ideal $P$ of $R$, let $E_R(R/P)$ denote the injective hull of $R/P$ over $R$. Then $G(E_R(R/P))=E_S (S/(P,x)S)$. Moreover,   
$M$ is an injective $R$-module if and only if $G(M)$ is an injective $S$-module. 
\end{Prop}

\begin{proof}
Let $d=\Dim (R_P)$.  
Since $R$ is a Gorenstein ring, $S_x/S$ a flat $R$-module, and $G(H^d_P(R)) \cong H^{d+1}_{(P,x)S}(S)$ by Lemma \ref{LemmaIdeal}, we have that
\begin{align*}
G(E_R(R/P))\cong G( H^d_{P R_P}(R_P))& \cong G( H^d_{P}(R)\otimes_R R_P)\\
& \cong G( H^d_P(R))\otimes_R R_P \cong  H^{d+1}_{(P,x)S}(S_P).
\end{align*} 
As $S_P/(P,x)S_P \cong R_P/PR_P$, $(P,x)S_P$ is a maximal ideal of the Gorenstein ring $S_P$, so
$$H^{d+1}_{(P,x)S}(S_P)=E_{S_P}\left(S_P/(P,x)S_P\right)=E_S\left(S/(P,x)S\right).$$ Therefore, $G\left(E_R(R/P)\right)=E_S \left(S/(P,x)S\right)$. Moreover, $G$ sends injective $R$-modules to 
injective $S$-modules because every
injective $R$-module is a direct sum of injective hulls of prime ideals.

It remains to prove that if $G(M)$ is an injective $S$-module, then $M$ is an injective $R$-module. This follows because $M=\Ann_{G(M)}(xS)$ by Theorem \ref{TeoEquiv}:  any injection of $R$-modules $\iota: N \hookrightarrow N'$ is also an injection of $S$-modules, where $x$ acts by zero.  Then any $S$-module map $f: N \to G(M)$ is an $R$-module map and must have image in $\Ann_{G(M)}(xS) = M$, so the induced map $g: N \to M$ is a map of $R$-modules such that $f = g \circ \iota$.
\end{proof}

\begin{Prop}\label{PropExt}
Let $R$ be a Gorenstein ring, and let $S = R[[x]].$
Since $R= S/xS$,  every $R$ module has an structure of $S$-module via extension of scalars. For $R$-modules $M, N$ and $i,j \in \NN$, $$\Ext^{i}_S (M,G(N))=\Ext^{i}_R (M,N).$$
\end{Prop}
\begin{proof}
Let $E^*=E^0\to E^1\to\ldots \to E^i\to \ldots $ be an injective $R$-resolution of $N$. Then
$G(E^*)$ is an injective $S$-resolution for $G(N)$ by Proposition \ref{PropInj}.
We notice that $\Hom_S (M,-)=\Hom_S(M,\Hom_S(R,-))$ as functors. Then 
\begin{align*}\Hom_S(M,G(E^*))=\Hom_S(M,\Hom_S(R,G(E^*))&=\Hom_S(M,E^*) =\Hom_R(M,E^*), \end{align*}
and the result follows.
\end{proof}
\section{Definitions and First Properties} \label{DefFirstProp}

\begin{Teo}\label{TeoEquivD} 
Let $K$ be a field, let $R=K[[x_1,\ldots, x_n]]$, and let $S=R[[x_{n+1}]]$. 
Let $\cC$ denote the category of $D(R,K)$-modules, and let $\cD$ denote the category of $D(S,K)$-modules that are supported on $\cV(xS)$. Then
\begin{enumerate}
\item[\emph{(i)}] $G:\cC\to\cD$ given by $G(M) = M \otimes_R S_{x_{n+1}}/S$ is an equivalence of categories with inverse $\widetilde{G}:\cD\to \cC$, where $\widetilde{G}(N)=\Ann_N (xS)$,
\item[\emph{(ii)}] $M$ is a finitely generated $D(R,K)$-module if and only if $G(M)$ is a finitely generated $D(S,K)$-module, and
\item[\emph{(iii)}] $\Length_{D(R,K)}M=\Length_{D(S,K)}G(M).$
\end{enumerate}
\end{Teo}

\begin{proof}
The proofs of the statements are analogous to the those of Theorem \ref{TeoEquiv}, Proposition \ref{PropFG}, and Corollary \ref{CorLen}, respectively.
\end{proof}

\begin{Teo}\label{WellDef}
Let $(R,m,K)$ be a local ring  containing a field, so that the completion $\widehat{R}$  admits a surjective ring map $\pi : S\surj \widehat{R}$, where $S=K[[x_1,\ldots, x_n]]$ for some $n \in \NN$.
Fix ideals $I_1,\ldots, I_s$ of $R$ and $j_1,\ldots, j_s \in \NN$.
Let $J_1,\ldots, J_s$ be the corresponding preimages of $I_1 \widehat{R},\ldots, I_s \widehat{R}$ 
in $S$. Then $$ \Length_{D(S,K)} H^{j_s}_{J_s}   \cdots H^{j_2}_{J_2}H^{n-j_1}_{J_1}(S)$$
is finite and depends only on $R$, $I_1,\ldots, I_s$ and $j_1,\ldots, j_s$, but neither on $S$ nor on $\pi$.
\end{Teo}
\begin{proof}
We may assume without loss of generality that $R$ is complete.
We know that $\Length_{D(S,K)} H^{j_s}_{J_s}  \ldots  H^{j_2}_{J_2}H^{n-j_1}_{J_1}(S)$ is finite by \cite[Corollary $6$]{Lyu2}. 
Let $\pi' : S'\to R$ be another surjection, where $S'=K[[y_1,\ldots,y_{n'}]]$.
Let $J'_1,\ldots, J'_s$ be the corresponding preimages of $I_1,\ldots, I_s $ 
in $S'$.

Let $S''=K[[z_1,\ldots,z_{n+n'}]]$. 
Let $\pi'': S''\to R$ be the surjection defined by $\pi''(z_j)=\pi(x_j)$ for $0\leq j\leq n$
and $\pi''(z_j)=\pi'(y_{j-n})$ for $n+1\leq j\leq n+n'$.
Let $J''_1,\ldots, J''_s$ be the corresponding preimages of $I_1 ,\ldots, I_s $ 
in $S''$ under $\pi''$. Let
 $\alpha:S\to S''$ be the map defined by $\alpha(x_j)=z_j$. We note  that $\pi''\alpha=\pi.$
There exists $f_1,\ldots, f_{n'}\in S$ such that $\pi''(z_{n+j})=\pi(f_j)$ for $j\leq n'.$
Then $z_{n+j}-\alpha(f_j)\in\Ker(\pi'').$ We note that $\beta:S''\to S $ defined by sending
$z_j\to x_j$ for $j\leq n$ and $z_{n+j}\to f_j$ for $j\leq n'$ is an splitting of $\alpha.$
Then
$J_i''=(\alpha(J_i), z_{n+1}-\alpha(f_1),\ldots, z_{n'+n}-\alpha(f_{n'}))S''$. Since 
$$z_1,\ldots, z_n, z_{n+1}-\alpha(f_1),\ldots, z_{n'+n}-\alpha(f_{n'})$$ form a regular system of parameters, we obtain that 
\begin{equation}
\label{J''LC} \Length_{D(S'',K)} H^{j_s}_{J''_s}   \ldots H^{j_2}_{J''_2} H^{n'+n-j_1}_{J''_1}(S'') 
= \Length_{D(S,K)} H^{j_s}_{J_s}  \ldots  H^{j_2}_{J_2}  H^{n-j_1}_{J_1}(S) 
\end{equation}
by Proposition \ref{PropIdeals} and Theorem \ref{TeoEquivD}. 
Similarly, $\Length_{D(S,K)} H^{j_s}_{J'_s}   \ldots  H^{j_2}_{J'_2} H^{n'-j_1}_{J'_1}(S')$ also equals \eqref{J''LC}, and  the result follows.
\end{proof}

\begin{Def}[Generalized Lyubeznik numbers] \label{NewLyuNums}
Let $(R,m,K)$ be a local ring  containing a field. 
Fix $I_1,\ldots, I_s$ ideals of $R$ and $i_1,\ldots, i_s\in \NN$.
Let $\pi$, $S$, $n$, and $J_1,\ldots, J_s$ as in Theorem \ref{WellDef}.
The following invariant, which is well defined by Theorem \ref{WellDef}, is called the \emph{generalized Lyubeznik number of $R$ with respect to $I_1,\ldots, I_s$ and $i_1,\ldots, i_s$}:
 $$\lambda^{i_s,\ldots, i_1}_{I_s,\ldots, I_1}(R):=\Length_{D(S,K)} H^{i_s}_{J_s}  \cdots  
H^{i_2}_{J_2}H^{n-i_1}_{J_1}(S).$$
\end{Def}

\begin{Rem}
In Definition \ref{NewLyuNums}, we may assume that $I_1\subseteq \ldots \subseteq I_\ell$, because
if an $R$-module $M$ is such that $H^0_I(M)=M$ for some ideal $I$ of $R$, then $H^i_{J}(M)=H^i_{I+J}(M)$ for every ideal
$J$ of $S.$ In addition, $\lambda^{i_s,\ldots, i_1}_{I_s,\ldots, I_1}(R)=\lambda^{i_s,\ldots, i_1}_{I_s ,\ldots,I_2, 0}(R/I_1).$
\end{Rem}

\begin{Prop}\label{PropRmk} 
If $(R,m,K)$ be a local ring containing a field, then $\lambda_{i,j} (R)=\lambda^{i,j}_{m,0} (R).$
\end{Prop}

\begin{proof}
Since completion is flat and the Bass numbers are not affected by completion, we may assume that $R$ is complete.
Take $S=K[[x_1,\ldots,x_n]]$ such that there exist a surjective ring map $\pi :S\surj R$.
Set $I=Ker(\pi)$, the preimage of the zero ideal in $R$. We notice that the maximal ideal, $\eta$, of $S$ is the preimage of
the maximal ideal, $m$, of $R$.
By \cite[Lemma $1.4$]{LyuDmodules},
$$
\lambda_{i,j} (R)=\dim_K\Ext^{i}_S(K,H^{n-j}_{I} (S))
=\dim_K\Hom_S (K,H^i_\eta H^{n-j}_{I} (S)).$$
Since $H^i_\eta H^{n-j}_{I} (S)$ is a finite direct sum of copies of $E_S(K)$ by \cite[Corollary 3.6]{LyuDmodules}, and $E_S(K)$ is a simple $D(S,K)$-module (cf.\ \cite{LyuInjDim}), we 
obtain that
$$
\dim_K\Hom_S(K,H^i_\eta H^{n-j}_{I} (S))
=
\Length_{D(S,K)}H^i_{\eta} H^{n-j}_{I}(S)
=\lambda^{i,j}_{m,0} (R),$$
and we are done.
\end{proof}

\begin{Rem}
In characteristic zero, \`Alvarez Montaner introduced a family of invariants using the multiplicities of 
the characteristic cycle of local cohomology modules \cite{AM-Proc}.
Like ours, this family includes the original Lyubeznik numbers; however, this definition does not include rings of prime characteristic.
\end{Rem}

\begin{Prop} \label{FirstProp}
Given $I_1 \subseteq \ldots \subseteq I_s$ ideals of a local ring $(R,m,K)$ containing a field, we have that
\begin{enumerate}
\item[\emph{(i)}] $\lambda^{i_s,\ldots, i_1}_{I_s,\ldots, I_1}(R)=0$ for $i_1>\dim(R/I_1)$,
\item[\emph{(ii)}] $\lambda^{i_s,\ldots, i_1}_{I_s,\ldots, I_1}(R)=0$ for $i_j>\dim(R/I_{j-1}), 2 \leq j\leq \ell,$ 
\item[\emph{(iii)}] $\lambda^{i_2,i_1}_{I_2,I_1}(R)=0$ for $i_2>i_1$,
\item[\emph{(iv)}] $\lambda^{i_1}_{I_1}(R)\neq 0$ for $i_1=\dim(R/I_1),$ and
\item[\emph{(v)}] $\lambda^{i_2,i_1}_{I_2,I_1}(R)\neq 0$ if $i_2=\dim(R/I_1)-\dim(R/I_2)$ and $i_1=\dim(R/I_1).$
\end{enumerate}
\end{Prop}
\begin{proof}
We may assume that $R$ is complete, so that it admits a surjective ring map $\pi : S\surj R$, where $S=K[[x_1,\ldots, x_n]]$ for some $n$.
Let $J_1,\ldots, J_s$ be the corresponding preimages of $I_1 ,\ldots, I_s$ 
in $S$.

As $S$ is Cohen-Macaulay, $\Depth_{J_1}(S)=\Codim(S/J_1)=n-\dim(S/J_1)=n-\dim(R/I_1)$, so that (i) and (iv) hold since $H^{i_1}_{J_1} (S)=0$ if $i<\Depth_{J_1}(S)$ and $H^{\Depth_I(S)}_{J_1}(S)\neq 0$.


To see (ii), note that   
\begin{align*}
\InjDim H^{i_{j-1}}_{J_{j-1}} \ldots H^{i_2}_{J_2} H^{n-i_1}_{J_1}(S) & \leq \dim(\Supp H^{i_{j-1}}_{J_{j-1}} \ldots H^{i_2}_{J_2} H^{n-i_1}_{J_1}(S))\\
& \leq \dim(S/J_{j-1})=\dim(R/I_{j-1})
\end{align*}
by \cite{Lyu2}. Similarly, (iii) follows because $\InjDim H^{n-i_1}_{J_1}(S)\leq \dim(\Supp H^{n-i_1}_{J_1}(S))\leq i_1.$


To prove (v), choose a minimal prime $P$ of $J_2$. Now,
$\Rad(J_1S_P)=PS_P$ in $S_P$. Then 
$H^{p}_{PS_P} H^{\dim(S_P)-q}_{J_1S_P}(S_P)\neq 0$ when
$p=q=\dim(S_P/J_1S_P)$ by \cite[Property $4.4$(iii)]{LyuDmodules}. Noting that
\begin{align*}
&\dim(S_P)=\dim(S)-\dim(S/P)=\dim(S)-\dim(S/J_1)=n-\dim(R/I_1), \text{ and}\\
&\dim(S_P/J_1S_P)=\dim(S/J_1)-\dim(S/J_2)=\dim(R/I_1)-\dim(R/I_2), \end{align*}
we see that $H^{i_2}_{J_2}H^{i_1}_{J_1}(S)\otimes_S S_P\neq 0$ if $i_2=\dim(R/I_1)-\dim(R/I_2)$ and $i_1=\dim(R/I_1)$.
\end{proof}

\begin{Lemma}\label{LemmaFieldExtension}
Given an extension of fields $K\subseteq L$, let $R=K[[x_1,\ldots,x_n]]$ and $S=L[[x_1,\ldots,x_n]].$
Let $R \to S $ denote the map induced by the field extension. 
If $M$ is a simple $D(R,K)$-module, then $M\otimes_R S$ is a simple $D(S,L)$-module.
\end{Lemma}
\begin{proof}
We have that $S=R \otimes_K L$ because the field extension is finite.  
Then $M\otimes_R S=M \otimes_K L$ and the action of $\partial \in D(S,L)$ is given by
$\partial (v\otimes a)=\partial(v)\otimes a.$
Let $e_1,\ldots,e_h$ be a basis for $L$ as $K$-vector space.
If $v\in M\otimes_K L$ is not zero, then $v=w_1 \otimes e_1+\ldots +w_h \otimes e_h$
for some $w_i\in M,$ where at least one $w_j$ is not zero.
We assume that $w_1\neq 0,$ an there exist operators $\delta_j\in D(R,K)$ such that
$w_j=\delta_j w_1$ because $M$ is simple.
Let $\delta=\delta_1+\ldots \delta_h$ and $u=e_1\ldots e_h.$
Then $v=\delta (w_1\otimes a)=a\delta (w_1\otimes 1).$
Since $v\neq 0,$ $\delta(w_1)\neq 0$ and there exist $\partial \in D(S,L)$
such that $\partial \delta w_1=w_1.$
Then $u^{-1}\partial v=w_1\otimes 1.$
Therefore for every $v\in M\otimes_K L$ not zero, $v\in D(S,L)\cdot w_1\otimes 1$ and $w_1\otimes 1\in D(S,L)\cdot v.$
Hence, $M\otimes_K L$ is a simple $D(S,L)$-module.
\end{proof}

\begin{Prop}\label{FieldExtension}
Let $K\inj L$ be a finite field extension, $R=K[[x_1,\ldots,x_n]],$ and $S=L[[x_1,\ldots,x_n]].$
Then for all ideal $I_1,\ldots, I_s$ of $R$ and all $i_1,\ldots, i_s\in \NN$,
$$
\lambda^{i_s,\ldots, i_1}_{I_s,\ldots, I_1}(R)= \lambda^{i_s,\ldots, i_1}_{I_s S,\ldots, I_1 S}(S).
$$
\end{Prop}
\begin{proof}
We have that $S=R\otimes_K L$ because the field extension is finite.  
Let 
$$
0=M_1\subsetneq \ldots \subsetneq M_\ell= H^{i_s}_{I_s}  \cdots
H^{i_2}_{I_2}H^{n-i_1}_{I_1}(S)
$$ 
be a filtration of $D(R,K)$-modules such that
$M_{i+1}/M_i$ is a simple $D(R,K)$-module.
Since $S$ is a faithfully flat $R$-algebra, $M_{i+1}/M_i\otimes_K L=M_{i+1}\otimes_K L/M_i\otimes_K L$ is a simple $D(S,L)$-module.
Thus, $\lambda^{i_s,\ldots, i_1}_{I_s,\ldots, I_1}(R)= \ell =\lambda^{i_s,\ldots, i_1}_{I_s S,\ldots, I_1 S}(S).$

\end{proof}

\begin{Prop} \label{CCineq}
Let $I_1,\ldots, I_\ell$ be ideals of $S=K[[x_1,\ldots,x_n]]$, where $K$ is a field of characteristic zero.  Then 
$\lambda^{i_\ell,\ldots, i_1}_{I_\ell,\ldots, I_1}(S)\leq e\left(H^{i_\ell}_{I_\ell}  \cdots  
H^{i_2}_{I_2}H^{n-i_1}_{I_1}(S)\right).$
\end{Prop}
\begin{proof}
Since $H^{i_\ell}_{I_\ell}  \cdots  
H^{i_2}_{I_2}H^{n-i_1}_{I_1}(S)$ is a holonomic $D(S,K)$-module, 
the claim follows from Remark \ref{PropMultiplicity}.
\end{proof}

For $R$ a one-dimensional or complete intersection ring, $\lambda_{i,j}(R)=1$ if $i=j=\dim R$, and vanishes otherwise.  However, Propositions \ref{PropCurves} and \ref{PropHyps} will show that the generalized Lyubeznik numbers capture finer information that can distinguish these cases.

\begin{Prop}\label{PropCurves}
Let $(R,m,K)$ be a complete local ring containing a field such that $\dim(R)=1$. Let $P_1,\ldots P_\ell$ be all the minimal primes of $R$. Then
$$\lambda^{1}_{0} (R)=\lambda^{1}_{0} (R/P_1)+\ldots +\lambda^{1}_{0} (R/P_\ell)+\ell-1.$$
\end{Prop}

\begin{proof}
We proceed by induction on $\ell$. Suppose $\ell=1$, and take a surjection 
$\pi: S = K[[x_1,\ldots, x_n]] \surj R \cong S/I$ where $I=\Ker(\pi)$. If $P$ is the minimal prime of $R$, then $\pi^{-1}(P)=\Rad(I)$ is the only minimal prime of $I$. Then
$$\lambda^{1}_{0} (R)=\Length_{D(S,K)}H^{n-1}_{I}(S)=\Length_{D(S,K)}H^{n-1}_{\pi^{-1}(P)}(S)=\lambda^{1}_{0} (R/P).$$

Now suppose that the formula holds for $\ell-1$.  There exists a surjection
$\pi :  S\surj R \cong S/I$, where $S=K[[x_1,\ldots, x_n]]$. Let $\eta$ denote the maximal ideal of $S$.
Let $Q_i=\pi^{-1}(P_i)$, so that $\Rad(I) = Q_1 \cap \ldots \cap Q_\ell$.  Let $J$ denote $Q_1\cap\cdots\cap Q_{\ell-1}$.  

Since $ \Rad(J + Q_\ell)=\eta$, the Mayer-Vietoris sequence in local cohomology with respect to $J$ and $Q_\ell$ gives the following exact sequence:
$$
0\to H^{n-1}_J (S)\oplus H^{n-1}_{Q_{\ell}}(S)\to H^{n-1}_{I}(S)\to H^{n}_{\eta}(S)\to 0.
$$
where $H^{n}_{I+J}(S)\cong E_S(K)$, a simple $D(S,K)$-module (cf.\ \cite{LyuInjDim}).
Then $\lambda^{1}_{0} (R)$ equals
\begin{align*}
\Length_{D(S,K)}H^{n-1}_{I}(S) &= \Length_{D(S,K)}H^{n-1}_{J}(S)+\Length_{D(S,K)}H^{n-1}_{Q_\ell}(S)+1\\
&=\lambda^{1}_{0} (S/J)+\lambda^{1}_{0} (S/Q_\ell)+1, \text{ and inductively,}\\
&= \left( \lambda^{1}_{0} (S/Q_1)+\ldots + \lambda^{1}_{0} (S/Q_\ell)+\ell -2 \right)+\lambda^{1}_{0} (S/Q_\ell)+1\\
&=\lambda^{1}_{0} (R/P_1)+\ldots +\lambda^{1}_{0} (R/P_{\ell})+\ell -1, \text{ as } R/P_i \cong S/Q_i.
\end{align*}
\end{proof}

\begin{Prop}\label{PropHyps}
Let $S=K[[x_1,\ldots, x_n]]$, where $K$ is a field. Let $f_1,\ldots ,f_\ell\in S$ be irreducible, and $f=f^{\alpha_1}_1\cdots f^{\alpha_\ell}_\ell$, where each $\alpha_i \in \NN$. Then
$$\lambda^{n-1}_{0} (S/f)\geq\lambda^{n-1}_{0} (S/f_1)+\ldots +\lambda^{n-1}_{0} (S/f_\ell)+\ell-1.$$
\end{Prop}
\begin{proof}

Since $ H^{i}_{I}(S)= H^{i}_{\sqrt{I}}(S)$ for every ideal $I\subseteq S$, we may assume that 
$\alpha_1=\ldots =\alpha_{\ell}=1$. Our proof will be by induction on $\ell$.  If $\ell=1$, it is clear.
We suppose that the formula holds for $\ell-1$ and we will prove it for $\ell$. Let 
$g=f_1\cdots f_{\ell-1}$. 
Since $f^{\alpha_\ell}_\ell,g$ form a regular sequence,  we obtain the exact sequence
$$
0\to H^{1}_{gS} (S)\oplus H^{1}_{f_{\ell}S}(S)\to H^{1}_{fS}(S)\to H^{2}_{(g,f_\ell)S}(S)\to 0
$$
by the Mayer-Vietoris sequence. Since $H^{2}_{(g,f_\ell)S}(S)\neq 0$, we have that \\$\Length_{D(S,K)}H^{2}_{(g,f_\ell)S}(S)\geq 1$. Moreover,
\begin{align*}
\lambda^{n-1}_{0} (S/fS)&=\Length_{D(S,K)} H^{1}_{fS}(S)  \\
&\geq\Length_{D(S,K)} H^{1}_{gS}(S) +\Length_{D(S,K)} H^{1}_{f_\ell S}(S) +1\\
&=\lambda^{n-1}_{0} (S/gS)+\lambda^{n-1}_{0} (S/f_\ell)+1, \text{ and inductively,}\\
&\geq\lambda^{n-1}_{0} (S/f_1) +  \ldots +\lambda^{n-1}_{0} (S/f_{\ell - 1} S)+\ell -2+\lambda^{n-1}_{0} (S/f_\ell S )+1\\
&=\lambda^{n-1}_{0} (S/f_1 S)+\ldots +\lambda^{n-1}_{0} (S/f_{\ell} S)+\ell -1.
\end{align*}
\end{proof}

\begin{Def}[Lyubeznik characteristic] \label{LyuCharc}
Let $(R,m,K)$ be a local ring containing a field such that $\dim(R)=d$. We define the \emph{Lyubeznik characteristic} of $R$ by
$$
\chi_\lambda(R)=\sum^d_{i=0}(-1)^i \lambda^i_{0}(R).
$$
\end{Def}

\begin{Prop}\label{M-VLyuChar}
Let $I$ and $J$ be ideals of a local ring $(R,m,K)$ containing a field. Then
$$\chi_\lambda(R/I)+\chi_\lambda(R/J)=\chi_\lambda(R/(I+J))+\chi_\lambda(R/I\cap J).$$
\end{Prop}
\begin{proof}
This an immediate consequence of the Mayer-Vietoris associated sequence for local cohomology with respect to $I$ and $J$.
\end{proof}

\begin{Prop} \label{CechCxAltSum}
If $I = (f_1, \ldots, f_\ell)$ an ideal of $S = K[[x_1, \ldots, x_n]]$, where $K$ is a field, then
$$\chi_\lambda(S/I)
=(-1)^n \sum \limits_{j=0}^\ell \sum \limits_{1 \leq i_1< \ldots < i_j \leq \ell} (-1)^{j} \lambda^{n-1}_0 \left( S/(f_{i_1} \cdot \ldots \cdot f_{i_j}) \right). $$  
In particular, if $f_1, \ldots, f_\ell$ form a regular sequence or if $\Char(k)=p>0$ and $S/I$ is a Cohen-Macaulay ring of dimension $d$, then $\lambda_0^{n-\ell}\left( S/(f_1, \ldots, f_\ell)S \right)$, or $\lambda_0^{d}\left( S/I \right)$, respectively, equals
$\sum \limits_{j=0}^\ell \sum \limits_{1 \leq i_1< \ldots < i_j \leq \ell} (-1)^{n-d+j} \lambda^{n-1}_0 \left(S/ (f_{i_1} \cdot \ldots \cdot f_{i_j}) \right).$
\end{Prop}

\begin{proof}
For brevity, let $D = D(S,K)$. By the additivity of $\Length_D(-)$ on short exact sequences and the \v{C}ech-like complex definition of local cohomology, $\sum \limits_{j=0}^\ell (-1)^j \LengthDshort \LC{j}{I}{S} 
=\sum \limits_{j=0}^\ell (-1)^j  \sum \limits_{1 \leq i_1 < \ldots < i_j \leq \ell} \LengthDshort S_{f_{i_1} \cdot \ldots \cdot f_{i_j} }.$  Moreover, 
the short exact sequence $0 \to S \to S_g \to H^1_{({f_{i_1} \cdot \ldots \cdot f_{i_j}})}(S) \to 0$ indicates that $\LengthDshort S_{f_{i_1} \cdot \ldots \cdot f_{i_j}} = \Length_D H^1_{(f_{i_1} \cdot \ldots \cdot f_{i_j})}+1$.  The first statement then follows from a straightforward calculation from the definition of Lyubeznik characteristic using these two observations.

The statement for a regular sequence is an immediate consequence, and the final statement follows since the only nonvanishing local cohomology module is $H^{n-d}_I(S)$ by \cite[Proposition $4.1$]{P-S}, since $S/I$ is Cohen-Macaulay.
\end{proof}

%
%

\section{Relations with $F$-rationality and $F$-regularity} \label{Frelations}

We recall Blickle's results \cite[Theorem $4.9$, Corollaries $4.10$ and $4.16$]{Manuel}.

\begin{Teo}[Blickle] \label{Manuel1}
Let $(S,m,K)$ be a regular local  $F$-finite ring of characteristic $p>0.$ Let $I$ be an ideal such that
$R= S/I$ is a domain of dimension $d$ and codimension $c$. Then $H^c_I (S)$ is a simple $D(S,\ZZ)$-module if and only if $0^*_{H^d_m(R)}$ is $F$-nilpotent.  As consequences,
\begin{itemize}
\item[\emph{(1)}]  If $R$ is $F$-rational, then $H^c_I (S)$ is a simple $D(S,\ZZ)$-module. If $R$ is $F$-injective, then $R$ is $F$-rational if and only if $H^c_I (S)$ is a simple $D(S,\ZZ)$-module.
\item[\emph{(2)}]  If $d=1$, then $H^c_I (S)$ is a simple $D(S,\ZZ)$-module if and only if
$R$ is unibranch.
\end{itemize}
\end{Teo}



These results indicate that the generalized Lyubeznik numbers detect $F$-regularity and $F$-rationality, as we see in the following proposition.

\begin{Prop}\label{F-regularity}
Let $(R,m,K)$ be a complete local domain of characteristic $p>0$ and of dimension $d$, such that $K$ is $F$-finite. 
The following hold.
\begin{itemize}
\item[\emph{(i)}] If $\lambda^d_0 (R)=1,$ then $0^*_{H^d_m(R)}$ is $F$-nilpotent.
\item[\emph{(ii)}] If $R$ is $F$-injective and $\lambda^d_0 (R)=1$, then $R$ is $F$-rational.
\end{itemize}
In addition, if $K$ is perfect, then:
\begin{itemize}
\item[\emph{(iii)}]  $\lambda^d_0 (R)=1$ if and only if $0^*_{H^d_m(R)}$ is $F$-nilpotent.
\item[\emph{(iv)}] If $R$ is $F$-rational, then  $\lambda^d_0 (R)=1$.
\item[\emph{(v)}] If $R$ is $F$-injective, then $\lambda^d_0 (R)=1$ if and only if $R$ is $F$-rational.
\end{itemize}
Moreover, if $R$ is one-dimensional, we have that:
\begin{itemize}
\item[\emph{(vi)}] If  $\lambda^d_0 (R)=1$, then $R$ is unibranch.
\item[\emph{(vii)}] If $K$ is perfect, then $\lambda^d_0 (R)=1$ if and only if $R$ is unibranch.
\end{itemize}
\end{Prop}
\begin{proof}
Take any surjective ring map $\pi : S\surj R$, where $S=K[[x_1,\ldots, x_n]]$, and let $I = \Ker(\pi)$.
Since $D(S,\ZZ)\subseteq D(S,K),$ $\Length_{D(S,K)} H^{n-d}_I(S)= \lambda^d_0(R)=1$ implies that $H^{n-d}_I(S)$ is a simple $D(S,\ZZ)$-module. Then (i) and (ii)
are consequences of the main statement and part (1) of Theorem \ref{Manuel1}, respectively.

If $K$ is perfect,  $D(S,\ZZ)= D(S,K)$ by \cite{Ye}, so $1 = \lambda^d_0(R)=\Length_{D(S,K)} H^{n-d}_I(S)$ precisely when $H^{n-d}_I(S)$ is a simple $D(S,\ZZ)$-module. Then (iii), (iv), and (v)
are consequences of the main statement and part (1) of Theorem \ref{Manuel1}.
Similarly, (vi) and (vii) follow from Theorem \ref{Manuel1} (3).
\end{proof}

%

\begin{Cor}\label{F-regularityGor}
Let $(R,m,K)$ be a complete local Gorenstein domain of characteristic $p>0$, of dimension $d$, and such that $K$ is $F$-finite. 
The following hold:
\begin{itemize}
\item[\emph{(i)}] If $R$ is $F$-pure and $\lambda^d_0 (R)=1$, then $R$ is $F$-regular.
\item[\emph{(ii)}] If $R$ is $F$-pure and $K$ is perfect, then 
$R$ is $F$-regular if and only if $\lambda^d_0 (R)=1$.
\end{itemize}
\end{Cor}
\begin{proof}
For a Gorenstein ring, $F$-rationality and  $F$-regularity are equivalent \cite{HoHu2}; additionally, $F$-injectivity and $F$-purity are equivalent \cite[Lemma 3.3]{Fedder}.  The result follows.
\end{proof}

\begin{Rem} \label{detcharp}
Let $R= K[X]$ be the polynomial ring over a perfect field $K$ of characteristic $p>0$ in the entries of an $r \times r$ matrix $X$ of indeterminates.  Let $m$ denote its homogeneous maximal ideal, and let $\Delta$ denote the principal ideal of $R$ generated by the determinant of $X$.  Then $R/\Delta$ is $F$-rational \cite[Theorem 9]{GS}, so by Proposition \ref{F-regularity} (iv), $\lambda_0^d(R_m/\Delta R_m) = 1.$
\end{Rem}

\begin{Rem}
In general, the Lyubeznik number $\lambda^d_0(R)$ is bounded by below by the number of minimal primes of $R$ that have dimension $d$.
Let $(R,m,K)$ be a complete local ring of dimension $d.$
Take any surjective ring map $\pi : S\surj R$, where $S=K[[x_1,\ldots, x_n]]$ for some $n$. Let $I$ denote the kernel of the surjection.
Let $P_1, \ldots  ,P_\ell$ be the minimal primes of $I.$ By iteratively using the Mayer-Vietoris sequence, we find that 
$H^d_{P_1}(S)\oplus\ldots \oplus H^d_{P_\ell}(S)\subseteq H^d_{I}(S).$
Therefore, $\lambda^d_0(R)\geq \ell.$

As a consequence, $R$ is a domain if it is equimensional and $\lambda^d_0(R)=1.$ 
Thus, several results of Proposition \ref{F-regularity} can be obtained by assuming only that $R$ is equidimensional.
\end{Rem}

\begin{Rem} \label{chains}

Let $I$ be an ideal of an $F$-finite regular local ring $S$, and suppose that the quotient ring $S / IS$ is $F$-pure.  Let $\tau_1$ denote the pullback of the test ideal of $S/I$ to $S$, and inductively let $\tau_i$ denote the pullback of the test ideal of the ring $S / \tau_{i-1}$  to $S$.  As demonstrated by Vassilev, the corresponding chain of ideals is of the form  \begin{equation} \label{Vchain: e} I \subsetneq \tau_1 \subsetneq \tau_1\subsetneq \ldots \subsetneq \tau_\ell = S\end{equation}  for some $\ell \geq 1$, and each quotient $S / \tau_i$ is $F$-pure \cite{Vass}. The following result, which connects this filtration with the generalized Lyubeznik numbers, is due to the first author and  P\'erez \cite{NuPe}:  If  $I = (f)$ is principal and $\ell$ is the length of the chain determined by the $\tau_i$ as in \eqref{Vchain: e}, then if $d = \dim(S/fS)$, $\lambda_0^{d}(S/fS) \geq \ell$.

By definition of the test ideals, we see that $\ell = 1$ if and only if the quotient $S/fS$ is $F$-regular, and so the inequality above shows that the generalized Lyubeznik number $\lambda_0^{d}(S/fS)$  must be large whenever $S/fS$ is ``far" from being $F$-regular.  This bound also shows that the hypersurface $S / fS $ must be $F$-regular if $\lambda_0^{d} (S/fS) = 1$; Corollary \ref{F-regularityGor} provides a partial converse to this statement.
\end{Rem}

\section{Generalized Lyubeznik Numbers of Ideals Generated by \\ Maximal Minors} \label{MaxMinors}

\begin{Lemma} \label{GenDMod}
Suppose that $K$ is a field of characteristic zero, $R = K[x_1, \ldots, x_n]$, and $S = K[[x_1, \ldots, x_n]]$.  
Let $f\in R$ be homogeneous. Let $D_R$ and $D_S$ denote $D(S,K)$ and $D(S,K),$ respectively.
If for some $N\in\NN,$ $D_S \frac{1}{f^N}= S_f,$ then $D_R \frac{1}{f^N} = R_f.$
\end{Lemma}
\begin{proof}
For every $r\in\NN,$ there exists 
$\delta=\sum \limits_{\alpha}g_\alpha \frac{\partial^\alpha}{\partial x^\alpha} \in D_S=
S\left \langle \frac{\partial}{\partial x_1},\ldots, \frac{\partial}{\partial x_n} \right\rangle $
such that $\delta \frac{1}{f^N}=\frac{1}{f^r}.$
In addition, there exist $\mu\in \NN$ and homogeneous $h_\alpha\in R$ such that $\mu>r$ and 
$\frac{\partial^\alpha}{\partial x^\alpha}\frac{1}{f^N}=\frac{h_\alpha}{f^\mu},$
so
$\delta  \frac{1}{f^N}=\sum_\alpha g_\alpha \frac{h_\alpha}{f^\mu}=\frac{1}{f^r}.$

We have that
$\sum \limits_\alpha g_\alpha h_\alpha=f^{\mu-r},$ and there exist homogeneous $g_{\alpha,t}\in R$ of degree $t$ such that $g_\alpha=\sum \limits^\infty_{t=0} g_{\alpha,t}.$
If $t_\alpha=(\mu-r)\deg(f)-\deg(h_\alpha),$ then 
$$f^{\mu-r}=\sum_\alpha g_\alpha h_\alpha
=\sum_\alpha \sum^\infty_{t=0} g_{\alpha,t} h_\alpha
=\sum_\alpha  g_{\alpha, t_\alpha} h_\alpha
$$
because $f$ and $h_\alpha$ are homogeneous polynomials.

Let $\widetilde{\delta}=\sum \limits_\alpha g_{\alpha,t_\alpha} \frac{\partial^\alpha}{\partial x^\alpha}\in D_R.$
Then
$$
\widetilde{\delta}\frac{1}{f^N}
=\sum_\alpha g_{\alpha,t_\alpha} \frac{\partial^\alpha}{\partial x^\alpha}\frac{1}{f^N}
=\sum_\alpha g_{\alpha,t_\alpha} \frac{h_\alpha}{f^\mu}\\
=\frac{\sum_\alpha g_{\alpha,t_\alpha} h_\alpha}{f^\mu}
=\frac{f^{\mu-r}}{f^\mu}
=\frac{1}{f^r}.
$$
Hence, $\frac{1}{f^r}\in D_R \frac{1}{f^N}$, and the result follows.
\end{proof}
\begin{Rem}
The conclusion of Lemma \ref{GenDMod} is not necessarily true if $f$ is not a homogeneous polynomial. 
Let $m$ denote the homogeneous maximal ideal of $R.$ 
If $f\in R$ is any polynomial such that
$R_m/fR_m$ is a regular local ring, then even if
$D(R,K)\frac{1}{f^N}\neq R_f,$ we have that 
$D(S,K)\frac{1}{f}=S_f.$
\end{Rem}
\begin{Rem}\label{BSpoly}
Let $b_f(s)$ denote the Bernstein-Sato polynomial of $f\in R$ over $R$ (cf. Section \ref{D-modules}).
If  $N=\max\{ j\in\NN \mid b_f(-j)=0\},$ then $D(R,K)\frac{1}{f^{N-1}}\neq R_f$ \cite[Lemma $1.3$]{WaltherBS}.  Therefore, if $f \in R$ is homogeneous,  $\Length_{D(S,K)} H^1_{(f)}(S) \geq  2$ by Lemma \ref{GenDMod}.
\end{Rem}

\begin{Ex} \label{CharDep}

Let $R = K[X]$ be the polynomial ring over a field $K$ in the entries of an $r \times r$ matrix $X$ of indeterminates, and let $m$ denote its homogeneous maximal ideal.   
Let $\Delta$ denote the principal ideal of $R$ generated by the determinant of $X$.
If $K$ has characteristic zero, the Bernstein-Sato polynomial of  the determinant of $X$ over $R$ is $b_{\operatorname{det}(X)}(s) = (s+1)(s+2) \cdots (s+r)$, so by Remark \ref{BSpoly}, $\lambda_0^{r^2-1} ( R_m/ \Delta R_m ) \geq 2.$  In stark contrast, by Remark \ref{detcharp}, if $K$ is instead a perfect field of characteristic $p>0$,  then $\lambda_0^{r^2-1}  \left(R_m/ \Delta R_m\right)=1$.  In particular, even when a specific Lyubeznik number is nonzero in both characteristic zero and characteristic $p>0$, their values may differ.

\end{Ex}

\begin{Ex} Now let $R$ be the polynomial ring over a field $K$ of characteristic zero in the entries of $X = [x_{ij}]$, an $r \times s$ matrix of indeterminates, where $r < s$.  Let $m$ denote its homogeneous maximal ideal, and let $I_t$ be the ideal generated by the $t \times t$ minors of $X$, and let $I = I_r$ be the ideal generated by the maximal minors of $X$.  By \cite[Theorem 1.1]{Witt}, $H^{r(s-r)+1}_I(R) \cong E_R(K)$, $0 \neq H^{i_t}_I(R) \hookrightarrow H^{i_t}_I(R)_{I_{t+1}} \cong E_R(R/I_{t+1})$ for $i_t = (r-t)(s-r)+1, 0 \leq t < r$, and all other $H^i_I(R) = 0.$  Thus, $\lambda^{r^2-1}_{0} (R_m/I R_m) = \lambda^{0, r^2-1}_{m,0}\left( R_m/I R_m\right) ( = \lambda_{0, r^2-1} (R_m/I  R_m)) = 1$, and $\lambda^{0 , i}_{m, 0} (R_m/I R_m) = 0$ for every $i \neq r^2-1$. 

Let $i_t = (r-t)(s-r)+1$, $t>0$, and suppose that $\lambda^{1,rs-i_t}_{m,0}(R_m/I R_m) = 0$.  Let $C$ be the cokernel of the injection $H^{i_t}_I(R) \hookrightarrow E_R(R/I_{t+1})$, so the short exact sequence $0 \to H^{i_t}_{I}(R) \to E_R(R/I_{t+1}) \to C \to 0$ gives rise to the long exact sequence in local cohomology: \[ \xymatrix{ 0\ar[r] &  H^0_m H^{i_t}_I(R) \ar[r] & H^0_m\left(E_R(R/I_{t+1})\right) \ar[r]& H^0_m\left(C\right)  \ar `[d] `[l] `[llld] `[ld] [lld] \\ & H^0_m H^{i_t}_I(R) \ar[r] & H^1_m\left(E_R(R/I_{t+1})\right) \ar[r] & \ldots } \]  Since the $I_{t+1}$ is the only associated prime of $E_R(R/I_{t+1})$ and of $H^{i_t}_I(R)$, \[ H^0_m H^{i_t}_I(R) =  H^0_m\left(E_R(R/I_{t+1})\right) = H^1_m\left(E_R(R/I_{t+1})\right) = 0,\] so $H^0_m\left(C\right) \cong  H^0_m H^{i_t}_I(R) = 0.$

%
%


If for some indeterminate $x_{\alpha \beta}$, the localization map $H^i_I(R) \to H^i_I(R)_{x_{\alpha \beta}}$ has a nonzero element $u$ in the kernel, then $x_{\alpha \beta}^N \cdot u = 0$ for some $N$.  But then, by symmetry, $x_{\alpha \beta}^N \cdot u = 0$ for \emph{all} indeterminates $x_{\alpha \beta}$, forcing every element of $H^{i_t}_I(R)$ to be killed by a power of $m$, a contradiction.  Similarly, the map $H^{i_t}_I(R)_{x_{11} } \to H^i_I(R)_{x_{11} \cdot x_{12}}$ is injective, and by induction, the composition of these localizations, $H^i_I(R) \to H^i_I(R)_{x_{1 1} x_{12} \cdot \ldots \cdot x_{r s}}$ will also be injective.  In particular, 
$
H^i_I(R)_{x_{\alpha \beta}} \hookrightarrow H^i_I(R)_{x_{1 1} \cdot x_{12} \cdot \ldots \cdot x_{r s}},\hbox{ and } 
\bigcap \limits_{\alpha,\beta} H^i_I(R)_{x_{\alpha \beta}} \hookrightarrow H^i_I(R)_{x_{1 1} \cdot x_{12} \cdot \ldots \cdot x_{r s}}.
$

Let $M$ denote $\bigcap \limits_{\alpha, \beta} H^i_I(R)_{x_{\alpha \beta}}$.  Since $x_{\alpha \beta} \notin I_{t+1}$ $H^i_I(R)_{I_{t+1}} \cong E_R(R/I_{t+1}),$ $M$ injects into $E_R(R/I_{t+1})$, and $M/H^i_I(R)$ injects into $E_R(R/I_{t+1})/H^i_I(R) = C$.  Since every element of $M/H^i_I(R)$ is killed by a power of $m$,  $M/H^i_I(R) = H^0_m\left(M/H^i_I(R))\right) \hookrightarrow H^0_m\left(C)\right) = 0$.  Thus, $M = H^i_I(R).$
\end{Ex}
\begin{Teo}
Continuing with the notation above,
if $r = 2$ and $s > 2$, then $$\lambda^{0,3}_{m,0}(R_m/I R_m) = \lambda^{s-1, s+1}_{m,0}(R_m/I R_m) = \lambda^{s+1, s+1}_{m,0}(R_m/I R_m)  = 1,$$ and all other $\lambda^{i,j}_{m,0}(R_m/I R_m)=0$.  In particular, each $\lambda^{1,i}_{m,0}(R_m/I R_m)=0$.
\end{Teo}

\begin{proof}
By \cite[Theorem 1.1]{Witt}, the only two nonzero local cohomology modules $\LC{i}{I}{R}$ are $\LC{2 s -3}{I}{R} \cong E_R (k)$ and $\LC{s-1}{I}{R} \inj E_R(R/I)$.  Replace $R$ by its localization at $m$, and consider the spectral sequence $E_2^{p,q} = H^{p}_{m} \LC{q}{I}{R} \overset{p}{\implies} \LC{p+q}{m}{R} = E_{\infty}^{p,q}$ \cite{Hartshorne}.  
Now, $H^{0}_{m} \LC{2 s - 3}{I}{R} \cong E_R(k)$ and $H^{p}_{m} \LC{2 s - 3}{I}{R} = 0$ for $p>0$.  In particular, $\lambda_{0,m}^{0,3}(R/I) = 1.$
Also note that $\dim R/I = s+1$, since if a $2 \times s$ matrix has vanishing $2 \times 2$ minors, the second row is a multiple of the first row.  Since $\Ass_R \LC{s-1}{I}{R} = \{I\}$, $H^{p}_{m} \LC{s-1}{I}{R} = 0$ for $p > s+1$.  These observations are indicated in Figure \ref{spectralseq}.

\begin{figure} \caption{$E_2^{p,q} = H^{p}_{m} \LC{q}{I}{R}$.}  \label{spectralseq} 

\begin{center}  
$\xymatrix@1@=0pt@M=1pt@R=1.1pt@C=0pt@!R{
&*+[F]{\boldsymbol{q}} &  \ar@{-}[ddddddddddd] \\
& & & & \vdots & \vdots & \vdots & & \vdots & \vdots & \vdots & \vdots & & \\
& \boldsymbol{2s-2} &  &  & 0  & 0  & 0 & \cdots & 0 & 0 & 0 & 0 & \cdots &  \\
& \boldsymbol{2s-3} & &  & E_R(k)  \ar@{->}[rrrrdddd]|{ d^{0, 2 s -3}_{s-1}}  & 0  \ar@{->}[rrrrdddd]|{ d^{1, 2 s -3}_{s-1}} & 0  \ar@{->}[rrrrdddd]|{ d^{2, 2 s -3}_{s-1}} & \cdots & 0 & 0 & 0 & 0 & \cdots & \\
& \boldsymbol{2s-4} & &  & 0 & 0 & 0 &  & 0 & 0 & 0 & 0 &\cdots & \\
& \vdots & & & \vdots  & \vdots & &  &  & \vdots & \vdots & \vdots & & \\
& \boldsymbol{s} &  & & 0 & 0 & 0  &  & 0 & 0 & 0 &0 & \cdots & \\
& \boldsymbol{s-1} & &  & E_2^{0,s-1} & E_2^{1,s-1} & E_2^{2,s-1}  &  \cdots &  E_2^{s-1,s-1}  &  E_2^{s,s-1}  & E_2^{s+1,s-1}  & 0 & \cdots \\
& \boldsymbol{s-2} &  & & 0 & 0 & 0 & \cdots & 0 &  0 & 0 & 0 &  \cdots & \\ 
 & \vdots & & & \vdots  & \vdots & \vdots & & \vdots & \vdots & \vdots & \vdots & & \\
 \ar@{-}[rrrrrrrrrrrrrr] & & & & & & & & & & & & & & & & & \\
& & & & \boldsymbol{0} & \boldsymbol{1} & \boldsymbol{2} &  \cdots &  \boldsymbol{s-1} & \boldsymbol{s} & \boldsymbol{s +1} & \boldsymbol{s+2} & & &*+[F]{\boldsymbol{p}} &}$  \end{center}

\end{figure}

As $\LC{2 s}{m}{R} \cong E_R(k)$ is the only nonzero local cohomology module of $R$ with support in $m$. The only possibly nonzero $E_2^{p,q} = H^{p}_{m}\LC{q}{I}{R}$ such that $p+q = 2 s$ is $H^{s+1}_{m} \LC{s-1}{I}{R}$, and so, since the spectral sequence maps to and from $H^{s+1}_{m} \LC{s-1}{I}{R}$ must all be zero (since the terms from which they come or go are zero), we must have that $H^{s+1}_{m} \LC{s-1}{I}{R}\cong E^{s+1, s-1}_{\infty} = E_R(k)$, so that, as $\dim R  - (s-1) = s+1$, $\lambda^{s+1, s+1}_{m,0}(R/I R) = 1$.  Moreover, every other $E^{p,q}_{\infty}$ must vanish.

Since $E_{s-1}^{0,2 s - 3} \cong E_R(k)$, we see that the sole differential that is (possibly) nonzero is $d^{0,2 s - 3}_{s-1} : E^{0 ,2 s - 3}_{s-1} \cong E_R(k) \to E^{s-1, s-1}_{s-1}.$ After taking cohomology with respect to the $d_{s-1}^{p,q}$ we must get zero at both the $(s-1, s-1)$ and $(0,2s-3)$ spots, so $d^{0, 2s-3}_{s-1}$ must be an isomorphism, and $H^{s-1}_{m} \LC{s-1}{I}{R} = E^{s-1, s-1}_{s-1} \cong  E_R(k)$, so that, as $\dim R - (s-1) = s+1$, $\lambda^{s-1, s+1}_{m,0}(R/I R) = 1$.  Since all other maps are the zero map, and after taking cohomology with respect to $d_{s-1}^{p,q}$ we must also get zero, all remaining local cohomology modules of the form $H^{p}_{I} \LC{q}{m}{R}$ must vanish (i.e., all except $p=0$, $q=3$, and $p = s-1$, $q=s-1$ and $p=s+1$, $q=s-1$), so that in these cases, $\lambda_{m,0}^{p,q}(R/I) = 0.$
\end{proof}

\section{Generalized Lyubeznik Numbers of Monomial Ideals} \label{Monomial}

In this section we characterize the generalized Lyubeznik numbers associated to monomial ideals. To do so, we 
make use of the categories of square-free and straight modules introduced by Yanagawa \cite{YSq,YStr}; we begin with some definitions and notation he first introduced.

\begin{Notation}
Let $S=K[x_1,\ldots,x_n]$, $K$ a field, and 
consider the natural $\NN^n$-grading on $S$.  For $\alpha=(\alpha_1,\ldots,\alpha_n)\in \ZZ^n$, 
we define $\Supp(\alpha)=\{i\mid \alpha_i > 0\}\subseteq [n]=\{ 1 ,\ldots,n\}$. For 
a monomial $x^\alpha=x_1^{\alpha_1}\cdots x_n^{\alpha_n}$, $\Supp(x^\alpha):=\Supp(\alpha)$. We say that $x^\alpha$ is \emph{square-free} if, for every $i \in [n]$, $\alpha_i$ either vanishes or equals one. 
Let $e_i$ denote the vector $(0,\ldots,0, 1, 0\ldots,0)\in\NN^n$, where ``$1$" is in the $i^\text{th}$ entry. 
If  $F\subseteq [n]$,  let $P_F$ denote the prime ideal generated by $\{x_i\mid i\not\in F\}$. 
If  $F\subseteq [n]$, we will often use $F$ instead of $\sum \limits_{i\in F} e_i$; for instance, $x^F$ denotes 
$\prod \limits_{i\in F} x_i$.  

Given a $\ZZ^n$-graded $S$-module $M$, and $\beta\in\ZZ$, $M(\beta)$ denotes the $\NN^n$-graded $S$-module
that has underlying $S$-module $M$, but with a shift in the grading: $M(\beta)_{\alpha}=M_{\alpha+\beta}$.  Let $\omega_S=S(-1,\ldots,-1)$ denote
the canonical module of $S$, and let {\bf *Mon} denote the category of $\ZZ^n$-graded $S$-modules. 
\end{Notation}  

\begin{Def}[Square-free monomial module]
An $\NN^n$-graded $S$-module $M=\bigoplus \limits_{\beta\in\NN^n} M_{\beta}$ is \emph{square-free} if it is finitely generated, and
the multiplication map
$M_\alpha\FDer{\cdot x_i} M_{\alpha+e_i}$ is bijective for all $\alpha\in\NN$, and all $i\in\Supp(\alpha)$. The category of square-free
$S$-modules is denoted {\bf Sq}, a subcategory of {\bf *Mon}.
\end{Def}

If $I$ is a square-free monomial ideal, then both $I$ and $S/I$ are square-free modules. 
Moreover, if $0\to M'\to M\to M''\to 0$ is a short exact sequence in {\bf *Mon}, then $M$ is a square-free module if and only if both $M'$ and 
$M''$ are square-free modules. 
In addition, if $M$ is a square-free module, then $\Ext^i_S(M, \omega_S)$ is a square-free module for every $i\in \NN$ \cite{YSq}.
Additionally, for any subset $G \subseteq F \subseteq [n]$, $S/P_F(-G)$ is a square-free module (where the grading of $P_F(-G)$ satisfies $[P_F(-G)]_{\ell} = [P_F]_{\ell-G}$). 

\begin{Rem} \label{simplesf}
An $\NN^n$-graded square-free $S$-module $M$ is a \emph{simple square-free
module} if it has no proper square-free non-trivial submodules. In fact, such a square-free module is simple if and only if it is isomorphic to $S/P_F (-F)$ for some $F\subseteq [n]$ \cite{YSq}.
\end{Rem}

\begin{Prop}[\text{\cite[Proposition 2.5]{YSq}}]\label{2.5Y}
An $\NN^n$-graded $S$-module $M$ is square-free if and only if
there exists a filtration of $\NN^n$-graded submodules $0 =M_0\subsetneq M_1\subsetneq \ldots\subsetneq M_t
= M$ such that, for each $i$ \textup{(}$0 \leq i \leq t-1$\textup{)}, $\overline{M}_i=M_i/M_{i+1}\cong S/P_{F_i}(-F_i)$ for some $F_i\subseteq [n]$ \textup{(}and so is, in particular, a simple square-free module\textup{)}.
\end{Prop}

As a consequence of Proposition \ref{2.5Y}, every square-free module $M$ has finite length in {\bf Sq}.  We now recall the following definition \cite{YStr}.

\begin{Def}[Straight module]
A  $\ZZ^n$-graded $S$-module $M=\bigoplus \limits_{\beta\in\ZZ^n} M_{\beta}$ is \emph{straight} if 
$\dim(M_\beta)<\infty$ for all $\beta\in \ZZ^n$, and
the multiplication map
$M_\alpha\FDer{\cdot x_i} M_{\alpha+e_i}$ is bijective for all $\alpha\in\ZZ^n$ and all $i\in\Supp(\alpha)$. 
The category of straight
$S$-modules is denoted {\bf Str}, a subcategory of {\bf *Mon}.
\end{Def}

\begin{Rem} \label{sqfreestraight}
If $M=\bigoplus \limits_{\beta\in\ZZ^n} M_{\beta}$ is a straight module, then $\overline{M}$ denotes the $\NN^n$-graded (square-free) submodule 
$\bigoplus \limits_{\beta\in\NN^n} M_{\beta}$. On the other hand, if $M$ is a square-free module, we can define the \emph{straight hull} of $M$, $\widetilde{M}$, as 
follows:  For $\alpha \in \NN^n$, let $\widetilde{M}_\alpha$ be a vector space isomorphic to $M_{{\scriptsize\Supp(\alpha)}}$, and let $\phi_\alpha : \widetilde{M}_\alpha \to M_{\Supp(\alpha)}$ 
denote such an isomorphism. 
Let $\beta=\alpha +e_i$ for some $i\in [n]$. If $\Supp(\alpha)=\Supp(\beta)$, we 
define $\widetilde{M}_\alpha \FDer{\cdot x_i} \widetilde{M}_\beta$  by the composition 
$\widetilde{M}_\alpha \FDer{\phi_\alpha} M_{{\scriptsize\Supp(\alpha)}} \FDer{\phi^{-1}_\beta} \widetilde{M}_\beta$; otherwise, we 
define $\widetilde{M}_\alpha \FDer{\cdot x_i} \widetilde{M}_\beta$  by the composition 
$\widetilde{M}_\alpha \FDer{\phi_\alpha} M_{{\scriptsize\Supp}(\alpha)} \FDer{x_i} M_{{\scriptsize\Supp(\beta)}}\FDer{\phi^{-1}_\beta} \widetilde{M}_\beta$.   Then $\widetilde{M}$ is straight, and its $\NN^n$-graded part is isomorphic to $M$.
\end{Rem}

\begin{Prop}[\text{\cite[Proposition 2.7]{YStr}}] \label{2.7Y} 
Continuing with the notation above, the functor ${\bf Str}\to {\bf Sq}$ defined by $M\to\overline{M}$ is an equivalence of categories
with inverse functor $N\to \widetilde{N}$.
\end{Prop}

\begin{Rem} \label{sthullsimple}
Let $L[F]$ denote the straight hull of $P_F(-F)$. By Proposition \ref{2.7Y} (noting Remark \ref{simplesf}), $L[F]$ is a simple 
straight module. We have that $L[F]_\alpha=k$ if $\Supp(\alpha)=F$, and is zero otherwise \cite{YStr}. Thus, $L[F]\cong H^\ell_{P_F}(\omega_S)$, where
$\ell=n-| F |$. 
\end{Rem}

\begin{Rem}\label{StrDMod}
Any straight module $M$ may be given the structure of a $D(S,K)$-module.
It suffices to define an action of  $\frac{1}{t!}\deriv{x_i}{t}$, for every $1\leq i \leq n$ and $t\geq1$:   Take $v\in M_\alpha$, where 
$\alpha=(\alpha_1,\ldots,\alpha_n)$. If $1 \leq \alpha_i \leq t$, we define $\frac{1}{t!}\deriv{x_i}{t} v=0$.
Otherwise, there exist $w\in M_{\alpha-te_i}$ such that 
$x^t_iw=v$, and we define $\frac{1}{t!}\deriv{x_i}{t} v=\binom{\alpha_i}{t} w$ if $\alpha_i>0$ and
$\frac{1}{t!}\deriv{x_i}{t} v=(-1)^{-\alpha_i+1}\binom{-\alpha_i}{t} w$ if $\alpha_i<0$. This observation extends in \cite[Remark 2.12]{YStr} to any field. We note that giving this $D(S,K)$-structure gives an exact faithful functor from  ${\bf Str}$ to the category of  $D(S,K)$-modules.
\end{Rem}
\begin{Teo} \label{sqfreeExt}
Let $K$ be a field, $S=K[x_1,\ldots,x_n]$, and $\widehat{S}=K[[x_1,\ldots,x_n]]$.
Let $I_1,\ldots,I_s\subseteq S$ be  ideals generated by square-free monomials.
Then 
\begin{align*} \lambda^{i_s,\ldots, i_1}_{I_s,\ldots,I_1} ( \widehat{S})
=\Length_{{\bf Str}} 
H^{i_s}_{I_s} \cdots H^{i_2}_{I_2} H^{n- i_1}_{I_1} (\omega_S) 
= \sum_{\alpha \in \{0,1\}^n} \dim_k \left[ H^{i_s}_{I_s} \cdots H^{i_2}_{I_2} H^{n-i_1}_{I_1} (\omega_S)\right]_{-\alpha}.
\end{align*}
Moreover, if $\Char(K)=0,$ then
$
\lambda^{i_1,\ldots, i_s}_{I_1,\ldots,I_s} (\widehat{S})
=\Mult (H^{i_s}_{I_s} \cdots H^{i_2}_{I_2} H^{n-i_1}_{I_1} (S)),
$ where $\Mult(-)$ denotes $D(S,K)$-module multiplicity \textup{(see Definition \ref{CCdef})}.
\end{Teo}
\begin{proof}
Let $M=H^{i_s}_{I_s} \cdots H^{i_2}_{I_2} H^{n-i_1}_{I_1} (S)$, so that $\lambda^{i_1,\ldots, i_s}_{I_1,\ldots,I_s} ( \widehat{S}) = \Length_{D(\widehat{S},K)}M$. By applying \cite[Corollary $3.3$]{YStr} iteratively, we see that  
$H^{i_s}_{I_s} \cdots H^{i_2}_{I_2} H^{n-i_1}_{I_1} (\omega_S)$
is an straight module.
By Propositions \ref{2.5Y} and \ref{2.7Y}, there is a strict ascending filtration of $\NN^n$-graded submodules 
$0 =M_0\subsetneq M_1\subsetneq \ldots\subsetneq M_t
= M$ 
such that each quotient $M_i/M_{i+1}$ is isomorphic to $\widetilde{P_{F_i} (-F_i)} \cong H^{ n-| F_i|}_{P_{F_i}}(\omega_S)$, 
and is also a filtration of $D(S,K)$-modules by Remark \ref{StrDMod}.
Moreover, 
$$0 =M_0\otimes_S \widehat{S}\subsetneq M_1 \otimes_S \widehat{S} \subsetneq \ldots\subsetneq M_t \otimes_S \widehat{S}= M\otimes_S \widehat{S}$$  
is a filtration of $D(\widehat{S},K)$-modules such that 
$
\left(\widetilde{M}_i \otimes_S \widehat{S}\right)/\left(\widetilde{M}_{i-1} \otimes_S \widehat{S}\right)\cong\widetilde{P_{F_i}(-F_i)} \otimes_S \widehat{S}\cong H^{n-| F_i|}_{P_{F_i}}(\widehat{S}).
$
Since $H^{n-| F_i|}_{P_{F_i}}(\widehat{S})$
is a simple $D(\widehat{S},K)$-module for every $F\subseteq [n]$,
$\Length_{D(\widehat{S},K)} M\otimes_S \widehat{S}= t$ as well.

If $K$ has characteristic zero, due to the filtration above and noting Remark \ref{PropMultiplicity},  
$$
CC(M)=\sum^t_{i=1} CC\left(\widetilde{M}_i / \widetilde{M}_{i-1}\right) =\sum^t_{i=1} CC\left(H^{ n-| F_i|}_{P_{F_i}}(S)\right),
$$ where $CC(-)$ denotes the characteristic cycle (see Definition \ref{CCdef}).
By \cite[Corollary $3.3$ and Remark $3.4$]{Montaner}, each $CC\left(H^{n-| F_i|}_{P_{F_i}}\right)= T^*_{ \{ x_i=0 |x_i\in P_{F_i} \} } \Spec(S)$. As a result, each $\Mult\left(H^{n-| F|}_{P_F}\right)=1$ and so $\Mult(M)=t$. Then
$ \lambda^{i_1,\ldots, i_s}_{I_1,\ldots,I_s} (\widehat{S})
=\Length_{{\bf Sq}} \overline{M}
=\Length_{{\bf Str}} M=e(M).$
\end{proof}
\begin{Rem}
The Lyubeznik numbers with respect to monomial ideals may depend on the field, as shown in \cite[Example $4.6$]{LyuNumMontaner}. 
\end{Rem}

\begin{Rem} \label{MonomialAlg}
For $K$ a field of characteristic zero, let $S=K[x_1,\ldots,x_n]$, and take $I\subseteq S$ an ideal generated by monomials.
Let $\widehat{S}=K[[x_1,\ldots,x_n]]$.
Combining work of \'Alvarez Montaner in \cite[Theorem $3.8$ and Algorithm $1$]{Montaner} with Theorem \ref{sqfreeExt} provides an algorithm to compute $\lambda^{i}_0 (\widehat{S}/I\widehat{S})$ in terms of $P_1, \ldots, P_N$, the minimal primes of $I$.  A consequence of this algorithm is the following inequality: \[\lambda^j_0(\widehat{S}/I\widehat{S}) \leq \sum \limits_{\ell=0}^N \sum \limits_{1 \leq i_1 < \ldots < i_\ell < N} \delta^j_{i_1, \ldots, i_\ell},\]  where $\delta^j_{i_1, \ldots, i_\ell} = 1$ if $\height(P_{i_1} + \ldots + P_{i_\ell}) = j + \ell -1$, and equals zero otherwise.  

\end{Rem}

\begin{Rem}
By Corollary \ref{MonomialAlg}, there is a straightforward algorithm to compute the $\lambda^{i}_0 (\widehat{S}/I\widehat{S})$ using the minimal primes of $I$.
\end{Rem}

\begin{Lemma} \label{LengthLoc}
Let $S = K[[x_1, \ldots, x_n]]$, $K$ a field. For a monomial $f$ with $|\Supp(f)| = j$, $\LengthD S_f  = 2^j.$
\end{Lemma}

\begin{proof}
By \ref{sqfreeExt}, $\LengthD \LC{1}{(x_{i_1} \cdot \ldots \cdot x_{i_j})}{S} = 2^j - 1$.  Since local cohomology is independent of radical, $\LC{1}{f}{S} = \LC{1}{(x_{i_1} \cdot \ldots \cdot x_{i_j})}{S} = 2^j-1$.   Due to the exact sequence $0 \to S \to S_{f} \to \LC{1}{f}{S} \to 0$ and the fact that $S$ is a simple $D(S,K)$-module, we have that $\LengthD S_f = \LengthD S + \LengthD \LC{1}{f}{S} = 2^j$.
\end{proof}

\begin{Prop} \label{lengthmonomial}
Let $K$ be a field, and let $S = K[[x_1, \ldots, x_n]]$, and let $I$ be an ideal of $S$ generated by square-free monomials $f_1, \ldots, f_\ell \in S$.  Then \[\chi_{\lambda} \left(S/I\right) = (-1)^n \sum_{j=0}^\ell \sum_{1 \leq i_1 < \ldots < i_j \leq \ell} (-1)^j 2^{\deg \lcm( f_{i_1}, \ldots,  f_{i_j} )}.\]
Moreover, if $S/I$ is also Cohen-Macaulay, the above equation equals $(-1)^d \lambda_0^d(S/I)$.  If, further, $f_1, \ldots, f_\ell$ form a regular sequence, this equals $(-1)^{n-1} \prod \limits_{i = 1}^\ell  (2^{\deg f_i}-1)^\ell$.
\end{Prop}
 
\begin{proof}  Since $| \Supp ( f_{i_1} \cdot \ldots \cdot f_{i_j})| = \deg \lcm ( f_{i_1}, \ldots, f_{i_j} )$, the first statement follows from Lemma \ref{LengthLoc} and Proposition \ref{CechCxAltSum}.  

If $S/I$ is Cohen-Macaulay, then by \cite[Proposition 3.1]{Montaner} (which is stated in characteristic zero, although the argument is characteristic independent), $H^j_I(S) = 0$ for all  $j \neq \height{I} = n-d$, and the statement follows.
If the $f_i$ also form a regular sequence, $\lcm(f_{i_1} \cdot \ldots \cdot f_{i_j}) = f_{i_1} \cdot \ldots \cdot f_{i_j}$ and $\deg (f_{i_1} \cdot \ldots \cdot f_{i_j} ) = \sum \limits_{r=1}^j \deg f_{i_r}$, and $\sum \limits_{j=0}^\ell \sum \limits_{1 \leq i_1 < \ldots < i_j \leq \ell} (-1)^j 2^{\left( \sum \limits_{r=1}^j \deg f_{i_r} \right)} =  \prod \limits_{i = 1}^\ell (1 - 2^{\deg f_i})^\ell = - \prod \limits_{i = 1}^\ell  (2^{\deg f_i}-1)^\ell$. 
\end{proof}

\subsection{Lyubeznik characteristic of Stanley-Reisner rings}
\begin{Def}[Simplicial complex, faces/simplices, dimension of a face, $i$-face, facet]
A \emph{simplicial complex} $\Delta$ on the vertex set
$[n] = \{1, \ldots, n\}$ is a collection of subsets, called \emph{faces} or \emph{simplices}, that are closed under
taking subsets. A face
$\sigma \in\Delta$ of cardinality $|\sigma| = i + 1$ is said to have \emph{dimension} $i$, and is called an $i$-\emph{face}
of $\Delta$. The \emph{dimension}  of $\Delta$, $\dim(\Delta),$ is the maximum of the dimensions of
its faces (or $-\infty$ if $\Delta=\varnothing$). We denote the set of faces of dimension $i$ of $\Delta$ by  
$F_i(\Delta)$.   
A face is a \emph{facet} if it is not contained in any other face.
\end{Def}
\begin{Rem} If $\Delta_1$ and $\Delta_2$ are simplicial complexes on the vertex set $[n]$, then
$\Delta_1\cap\Delta_2$ and $\Delta_1\cup\Delta_2$ are also simplicial complexes.
\end{Rem}
\begin{Def}[Simple simplicial complex]
We say that a simplicial complex $\Delta$ on the vertex set
$[n]$ is \emph{simple }if it is equal to $\cP (\sigma)$, the power set of a subset $\sigma$ of $[n]$. 
\end{Def}
\begin{Rem}\label{facets}
If $\sigma_1,\ldots,\sigma_\ell$ are the maximal facets of $\Delta,$ then $\Delta=\cP(\sigma_1)\cup\ldots\cup \cP(\sigma_\ell).$ In particular, a simplicial complex
is determined by its facets. 
\end{Rem}
\begin{Notation}
If $\Delta$ is a simplicial complex on the vertex set $[n]$ and $\sigma\in\Delta,$ then $x^\sigma$ denotes $\prod \limits_{i\in\sigma} x_i\in K[x_1,\ldots,x_n]$.
\end{Notation}
\begin{Def}[Stanley-Reisner ideal of a simplicial complex] The \emph{Stanley-Reisner ideal of the simplicial complex $\Delta$}
is the square-free monomial ideal
$I_{\Delta} = ( x^\sigma \mid \sigma\not\in\Delta )$ of $K[x_1,\ldots,x_n]$. The \emph{Stanley–-Reisner ring of $\Delta$} is $K[x_1,\ldots,x_n]/I_\Delta$.
\end{Def}
\begin{Teo}[\textup{\cite[Theorem 1.7]{MillerSturmfels}}]\label{Bijection}
The correspondence $\Delta\mapsto I_\Delta$ defines a bijection from
simplicial complexes on the vertex set $[n]$ to square-free monomial ideals
of $K[x_1, \ldots, x_n].$ Furthermore,
$
I_\Delta=\bigcap \limits_{\sigma\in\Delta} ( x^{[n]\setminus \sigma}).
$
\end{Teo}
\begin{Prop}\label{OpSimp}
Under the correspondence in Theorem \ref{Bijection}, 
$I_{\Delta_1\cap\Delta_2} =I_{\Delta_1}+I_{\Delta_2}$ and 
$I_{\Delta_1\cup\Delta_2}=I_{\Delta_1}\cap I_{\Delta_2}$
for all simplicial complexes 
$\Delta_1$ and $\Delta_2$.
\end{Prop}
\begin{proof}
For the first statement, we see that
\begin{align*}
x^\sigma \in I_{\Delta_1\cap\Delta_2}  &\Leftrightarrow \sigma\not\in \Delta_1\cap\Delta_2
\Leftrightarrow \sigma\not\in \Delta_1\hbox{ or }\sigma\not \in\Delta_2 \\
& \Leftrightarrow x^\sigma\in I_{\Delta _1}\hbox{ or }x^\sigma\in I_{\Delta_1}
\Leftrightarrow x^\sigma \in I_{\Delta_1}+I_{\Delta_2}.
\end{align*}
The proof of the second statement is analogous.
\end{proof}

\begin{Teo}\label{LyuCharSR}
Take a simplicial complex $\Delta$ on the vertex set $[n]$. Let $R$ be the Stanley-Reisner ring of $\Delta$, and let $m$ be its maximal homogeneous ideal.
Then
$$
\chi_\lambda(R_m)=\sum^{n}_{i=-1} (-2)^{i+1} |F_i (\Delta)|.
$$
\end{Teo}

\begin{proof}
Let $S=K[x_1,\ldots, x_n]$, and let $\eta$ be its maximal homogeneous ideal.
We proceed by induction on $d:=\dim(\Delta).$ If $d=0$, then $\Delta=\{\varnothing\}.$ Then $I_\Delta=\eta,$ and $R=K$, so that
$\chi_\lambda(R_m)=1=(-2)^0=\sum \limits^{n}_{i=-1} (-2)^{i+1} |F_i (\Delta)|.$

Assume that the formula holds for all simplicial complexes of dimension less or equal to $d$.
Take a simplicial complex $\Delta$ of dimension $d+1$.
Consider all its facets, $\sigma_1,\ldots, \sigma_\ell.$ We now proceed by induction on $\ell$.
If $\ell=1$,  suppose that $\Delta_1=\cP (\sigma_1)$, where $\sigma_1=\{i_1,\ldots,i_j\}$ and $\dim(\sigma_1)=j.$ 
Then $I_{\Delta_1}=(x_i\mid i\not\in \sigma_1)S$, $R \cong K[x_1, \ldots, x_{n-j}]$, and 
\begin{align*}
\chi_\lambda(R_m)&=\Length_{D(\widehat{S_\eta},K)} H^j_{I_{\Delta_1}}(\widehat{S_\eta})=(-1)^{j}=(1-2)^{j}\\
&=\sum^{j}_{k=0} 1^{j-k}(-2)^k\binom{j}{k}=\sum^{j-1}_{k=-1} (-2)^{k+1}\binom{j}{k+1}=\sum^{j-1}_{k=-1} (-2)^{k+1}|F_k (\Delta)|.
\end{align*}

Assume that the formula is true for simplicial complexes of dimension $d+1$ with $\ell$ facets, and take a simplicial complex $\Delta$ of dimension $d+1$ with $\ell+1$ facets, $\sigma_1, \ldots, \sigma_\ell$.
Let $\Delta_i=\cP(\sigma_i)$ and $\Delta '=\Delta_1\cup\ldots \cup \Delta_\ell$. 
Then $\Delta=\Delta '\cup\Delta_{\ell+1}$.
We may assume, by renumbering, that $\dim(\Delta_\ell)=\dim(\Delta)$. Then $\dim(\Delta'\cap\Delta_\ell)<\dim(\Delta_\ell)$ by our choice of $\Delta_\ell$
and as we chose the decomposition given by the maximal facets. 
Therefore $\Lyuchar{R_m}$ equals
\begin{align*}
\Lyuchar{(S/I_{\Delta'\cup\Delta_\ell})_\eta} &=\Lyuchar{(S/I_{\Delta'}\cap I_{\Delta_\ell})_\eta} \hbox{ by Proposition \ref{OpSimp}} \\
&= \Lyuchar{(S/I_{\Delta'})_\eta)+ \chi_\lambda((S/I_{\Delta_\ell})_\eta}
-\Lyuchar{(S/(I_{\Delta'}+I_{\Delta_\ell}))_\eta} \hbox{ by Proposition \ref{M-VLyuChar}}\\
&=\Lyuchar{(S/I_{\Delta'})_\eta)+ \chi_\lambda((S/I_{\Delta})_\ell)_\eta}
-\Lyuchar{(S/(I_{\Delta'\cap\Delta_\ell}))_\eta} \hbox{ by Proposition \ref{OpSimp}}\\
&=\sum^{n}_{i=-1} (-2)^{i+1} | F_i (\Delta')| +\sum^{n}_{i=-1} (-2)^{i+1} | F_i (\Delta_\ell)|-\sum^{n}_{i=-1} (-2)^{i+1} | F_i (\Delta'\cap\Delta_\ell)|\\
&=\sum^{n}_{i=-1} (-2)^{i+1} (| F_i (\Delta')|+ | F_i (\Delta_\ell)|- | F_i (\Delta'\cap\Delta_\ell)| )\\
&=\sum^{n}_{i=-1} (-2)^{i+1} | F_i (\Delta'\cup \Delta_\ell) | =\sum^{n}_{i=-1} (-2)^{i+1} | F_i ( \Delta)|.
\end{align*}
\end{proof}

The above computation is related to work in \cite{MontanerAdv}. 

\begin{Ex} \label{SRex}
Let $K$ be a field, $S = K[x_1, x_2, x_3, x_4, x_5]$, and $m = (x_1, x_2, x_3, x_4, x_5)$.  Consider the ideal $I =(x_1 x_3, x_1 x_4, x_2 x_3, x_2 x_4, x_2 x_5) = (x_1, x_2, x_5) \cap (x_3, x_4, x_5)$ of $S$.  Note that $R := S/I$ is the Stanley-Reisner ring of the simplicial complex in Figure \ref{simplicialComplex}.
\begin{figure}
\begin{center}
\begin{pspicture}(3, 2.5)
\rput(-0.25,0.25){$x_2$}
\rput(0.75,1.25){$x_1$}
\rput(1.75,0.65){$x_5$}
\rput(3.25,1.25){$x_3$}
\rput(4.25,0.25){$x_4$}
\psline(1,1)(0,0)(2,0.5)
\pscustom[fillstyle=solid, fillcolor=black, opacity=0.4]{
\psline(2,0.5)(3,1)(4,0)(2,0.5)}
\psdot(0,0)
\psdot(2, 0.5)
\psdot(4,0)
\psdot(1,1)
\psdot(3,1)
\end{pspicture}
\end{center}
\caption{Simplicial complex in Example \ref{SRex}.}
\label{simplicialComplex}
\end{figure}
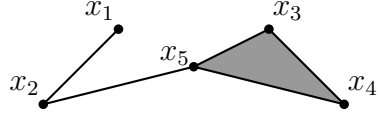
Using Theorem \ref{LyuCharSR}, we see that $\chi_\lambda(R_m) = 5 + (-2)\cdot5 + 4 \cdot 1 = -1.$  Moreover, if $K$ has characteristic zero, Corollary \ref{MonomialAlg} implies that $\lambda^3_0(\widehat{S}/I\widehat{S}) = 2$, $\lambda^4_0(\widehat{S}/I\widehat{S}) = 1$, and all other $\lambda^j_0(\widehat{S}/I\widehat{S}) = 0$, confirming the calculation of the Lyubeznik characteristic.  
\end{Ex}

\begin{Rem} \label{LyuCharChar}
In characteristic zero, \`Alvarez Montaner has given formulas for $|F_k (\Delta)|$ in terms of the characteristic cycle multiplicities of 
$H^1_{I_\Delta}(K[x_1,\ldots, x_n])$ (cf.\ \cite[Proposition $6.2$]{Montaner}) 
\end{Rem}

\begin{Rem}
Theorem \ref{LyuCharSR} shows that the Lyubeznik characteristic of Stanley-Reisner rings does not depend on their characteristic, although their Lyubeznik numbers do have such a dependence (cf.\ \cite[Example $4.6$]{LyuNumMontaner}).
\end{Rem}
  
\section*{Acknowledgments}
Many thanks go to Josep \`Alvarez Montaner, Xavier G\'omez-Mont,
Daniel Hern\'andez, Mel Hochster, Gennady Lyubeznik, and Felipe P\'erez for useful mathematical conversations related to this work.
We are also grateful to the American Mathematical Society and their Mathematical Research
Communities program for supporting meetings between the authors; we also thank the organizers of the aforementioned program.
The first author also thanks the National Council of Science and Technology of Mexico for its support through grant $210916.$

\bibliographystyle{alpha}
\bibliography{References}

\small{
{\sc Department of Mathematics, University of Michigan, Ann Arbor, MI $48109$-$1043,$ USA.}

\vspace{.2cm}

{\it Email address:}  \texttt{luisnub@umich.edu}

\vspace{.5cm}

{\sc Department of Mathematics, University of Minnesota,  Minneapolis, MN $55455,$ USA.} 

\vspace{.2cm}

{\it Email address:}  \texttt{ewitt@umn.edu}
}

\end{document}